\documentclass[11pt]{amsart}
\usepackage{amssymb, amsmath, amsthm, mathrsfs, amsopn, bm, color}
\usepackage{amsrefs, framed}

\usepackage[a4paper, centering]{geometry}
\usepackage{bookmark} 
\usepackage{hyperref} 
\usepackage[noabbrev,capitalize]{cleveref}

\def\DM{{\mathcal{DM}^{\infty}}}
\def\DMA{{{\rm DM}^\infty}}

\newcommand*{\DMloc}{\mathcal{DM}^{\infty}_{{\rm loc}}}
\newcommand*{\DMAloc}{{\rm DM}^\infty_{{\rm loc}}}

\DeclareMathOperator{\Tr}{Tr}
\newcommand*{\Trace}[3][\pm]{\Tr^{#1}(#2, #3)}
\newcommand*{\Trp}[2]{\Trace[+]{#1}{#2}}
\newcommand*{\Trm}[2]{\Trace[-]{#1}{#2}}

\newcommand*{\jump}[1]{\Theta_{#1}}
\DeclareMathOperator{\Trv}{\bf{Tr}}
\newcommand*{\Tracev}[3][\pm]{\Trv^{#1}(#2, #3)}
\newcommand*{\Trpv}[2]{\Tracev[+]{#1}{#2}}
\newcommand*{\Trstarv}[2]{\Tracev[*]{#1}{#2}}
\newcommand*{\Trmv}[2]{\Tracev[-]{#1}{#2}}
\newcommand*{\Trpmv}[2]{\Tracev[\pm]{#1}{#2}}

\def\nuint{\widetilde{\boldsymbol{\nu}}}

\newcommand{\A}{\boldsymbol{A}}
\newcommand{\B}{\boldsymbol B}
\newcommand{\C}{\boldsymbol C}

\DeclareMathOperator{\DIV}{{\rm Div}}

\newcommand{\bfu}{{\bf u}}
\def\R{\mathbb{R}}

\def\N{\mathbb{N}}

\def\MA{{\bf A}}

\DeclareMathOperator{\Div}{div}

\newcommand{\medint}{-\kern  -,375cm\int}
\newcommand{\medintinrigo}{-\kern  -,315cm\int}

\newcommand{\eps}{\varepsilon}

\newcommand{\LLN}{{\mathcal L}^N}

\newcommand{\res}{\mathop{\hbox{\vrule height 7pt width .5pt depth 0pt
\vrule height .5pt width 6pt depth 0pt}}\nolimits} % macro per la restrizione
\def\pscal#1#2{\left\langle #1\,,\, #2 \right\rangle}

\newcommand{\ban}[1]{\left\langle  #1 \right\rangle}  % <  > 

\newcommand{\Haus}[1]{{\mathscr H}^{#1}} % Misura di Hausdorff
\newcommand{\Leb}[1]{{\mathscr L}^{#1}} % Misura di Lebesgue

\newcommand{\BorelN}{\mathscr{B}(\Omega;\R)}

\newtheorem{definition}{Definition}[section]

\newtheorem{lemma}[definition]{Lemma}

\newtheorem{theorem}[definition]{Theorem}

\newtheorem{proposition}[definition]{Proposition}

\newtheorem{corollary}[definition]{Corollary}
\theoremstyle{remark}
\newtheorem{remark}[definition]{Remark}

\makeatletter
\def\@settitle{\begin{center}%
		\baselineskip14\p@\relax
		\bfseries
		\uppercasenonmath\@title
		\@title
		\ifx\@subtitle\@empty\else
		\\[5ex]%\@subtitle
		\normalsize\mdseries\@subtitle
		\fi
	\end{center}%
}
\def\subtitle#1{\gdef\@subtitle{#1}}
\def\@subtitle{}
\makeatother

\numberwithin{equation}{section}

\begin{document}
\title[Gauss--Green formulas for divergence measure tensor fields]
{Gauss--Green formulas for divergence measure \\ tensor fields on rough domains}

\author[V.~De Cicco]{Virginia De Cicco}
\address{Department of Basic and Applied Sciences for Engineering, Sapienza University of Rome \\
	Via A.\ Scarpa 10 -- I-00185 Rome (Italy)}
\email{virginia.decicco@uniroma1.it}
\author[G.~Scilla]{Giovanni Scilla}
\address{Department of Mathematics and Applications ``R. Caccioppoli'', University of Naples Federico II\\
	Via Cintia, Monte Sant'Angelo -- I-80126 Naples (Italy)}
\email{giovanni.scilla@unina.it}

\thanks{\textit{Acknowledgments}.
The authors are members of  the Istituto Nazionale di Alta Matematica (INdAM), Gruppo Nazionale per l'Analisi Matematica, la Probabilità e le loro Applicazioni (GNAMPA)}

\maketitle

\begin{abstract}
We introduce a notion of pairing between essentially bounded tensor fields with divergence measure and vector-valued functions of bounded variation, extending the classical theory to the tensorial setting. This naturally leads to an adaptation of the definition of normal trace for tensor fields with measure divergence even on a rectifiable set.
As a consequence, we establish tensorial Gauss–Green formulas that remain valid on sets with low regularity, including sets of finite perimeter. These results yield a unified and robust framework for integration by parts in the presence of irregular tensor fields and domains. 
\end{abstract}

\section{Introduction}

In continuum mechanics and in many problems of mathematical analysis, the following classical Divergence Theorem and the Gauss–Green formula in tensor form play a fundamental role (see, e.g., \cite{Goodbody, Gurtin1973}).
\begin{theorem}[Divergence Theorem in tensor form]\label{diverg}
Let $\Omega$ be a bounded open subset of $\R^N$ with smooth boundary. Let ${\MA}=(A_{ij})\in C^1(\Omega;\mathbb{R}^{N\times N})$, i.e. ${\MA}$ is a second-order regular tensor field. Then
\begin{equation*}
\int_{\Omega} \DIV{\MA} \,dV= \int_{\partial \Omega} {\MA}\, {\bf n}\, dS\,,
\end{equation*}
where 
\begin{equation*}\label{lllll}
\DIV{\MA} = \left(\sum_{j=1}^N \partial_jA_{ij}\right)_{i=1,\dots,N}\,, \quad \mbox{ and }\quad {\MA}\, {\bf n} =\left(\sum_{j=1}^NA_{ij}n_j\right)_{i=1,\dots,N}\,,
\end{equation*}
$\mathbb{R}^{N\times N}$ denotes the set of $N\times N$ matrices, ${\bf n}=(n_1,\dots,n_N)$ is the exterior normal at the boundary $\partial \Omega$ of $\Omega$, $dV$ is the volume element on $\Omega$ and $dS$ is the area element on $\partial\Omega$. 
\end{theorem}

The tensor ﬁeld $\MA$ is commonly interpreted as the {\it {Cauchy stress tensor}}.

\begin{corollary}[Gauss--Green formula in tensor form]\label{gaussclassica}
Let ${\MA}=(A_{ij})\in C^1(\Omega;\mathbb{R}^{N\times N})$ and $\bfu\in C^1(\Omega;\R^N)$. Then the following pointwise identity holds:
\begin{equation}\label{bbbb}
\bfu \cdot \DIV{\MA}+{\MA}:\nabla\bfu=  \Div({\MA}^T\bfu)\,,\ \ \mbox{ in } \Omega\,.
\end{equation}
Moreover, we have
\begin{equation}
\int_{\Omega} \bfu \cdot \DIV{\MA} \,dV+ \int_{\Omega}{\MA}:\nabla\bfu \,dV = \int_{\partial \Omega} \bfu\cdot({\MA}\, {\bf n})\, dS\,.
\label{eq:regularform}
\end{equation}
\end{corollary}
For Theorem \ref{diverg}, see for instance \cite[Section 6(1)]{Gurtin1973}, while the identities \eqref{bbbb} and \eqref{eq:regularform} are stated in \cite[Section 4(1), eq. (11), and Section 18(1)]{Gurtin1973}.

Several foundational issues in continuum mechanics--particularly those related to the representation of Cauchy fluxes--have motivated the development of weaker or generalized forms of the Gauss–Green formulas. Considerable effort has been devoted to establishing these results in more general settings and under minimal regularity assumptions.

For a detailed historical overview of divergence theorems we refer to \cite{ChTo3}. See also the recent and comprehensive treatment in \cite{SchoSchu}.
In the papers \cite{Anz,Cas,ChenFrid,ChFr1,ChTo2,ChTo,ChToZi,CCDM,CD3,CDCM,Schu,Silh1991} the study is confined to the case of vector fields $\A$ and scalar functions $\bfu$. For the case of tensor fields $\mathbf A$ and vector functions $\mathbf u$, we refer to \cite{DGMM,LuSiZa,Silh}. 

In particular, in \cite{LuSiZa} a second-order \emph{measure-valued} tensor field $\mathbf A$ is considered, namely a stress field represented by a tensor-valued Radon measure whose distributional divergence is a vector-valued measure. In that framework, the following notion is introduced: a tensor $\mathbf A$ is said to be \emph{equilibrated} if there exists a vector–valued Radon measure $\boldsymbol\mu := \DIV\mathbf A$ in an open set $\Omega$, called the divergence of $\mathbf A$, and a measure on $\partial\Omega$, called the normal trace of $\mathbf A$, such that the Gauss–Green formula \eqref{eq:regularform} holds for every sufficiently regular vector-valued function $\mathbf u$. An equilibrated tensor field admits a natural mechanical interpretation: it represents the stress field in a continuous body acted upon by a body force described by a measure $\boldsymbol\mu$ in $\Omega$ and by a boundary traction given in the sense of measures on $\partial\Omega$.

The aim of this paper is to identify a broad class of equilibrated tensor fields, namely the class of essentially bounded matrix-valued fields with measure divergence
$$\DMA(\Omega;\mathbb{R}^{N\times N}) :=\{ \MA \in L^\infty(\Omega;\mathbb{R}^{N\times N}) : \DIV\MA\in \mathcal M(\Omega;\R^N)
\},$$
and to establish Gauss--Green formulas for such tensors under minimal regularity assumptions.
We restrict our attention to essentially bounded fields, although $L^p$, $1\leq p <+\infty$, and measure-valued tensor fields have also been investigated in the literature (see \cite{CDCS,LuSiZa,Silh}). The analysis of these more general settings will be addressed in a forthcoming contribution.

Divergence-measure vector ﬁelds arise naturally in the context of nonlinear hyperbolic equations.
The theory of these fields is a fundamental issue in order introduce new techniques in entropy methods (see  \cite{ChenFrid,ChFr1,ChTo2,ChTo,ChToZi,CDC2,CDD,CDDG,Pan1}).
In particular, much of the recent progress in this area stems from applications to the uniqueness theory of weak entropy or kinetic solutions to hyperbolic conservation laws. Divergence-measure matrix-valued fields provide the appropriate framework for the analysis of hyperbolic systems of conservation laws and balance laws (see \cite[Section 5]{ChTo2} and \cite[Sections 10–11]{ChToZi}). The results established in the present paper may contribute to the development of an entropy approach for systems as well.

Another motivation for our study arises from elasticity and plasticity theory. In this framework one investigates the deformation $\mathbf u$ of an elastic body occupying a domain in $\mathbb{R}^N$, clamped along a portion of its boundary and subjected to a prescribed vector field of external forces (see, for instance, \cite[Sect.~6.6]{ABM} for the linearized elasticity system).

For this reason, in the present paper we allow for very irregular strain functions $\mathbf u$, belonging to the class $BV$ of functions of bounded variation. In a subsequent work, we plan to extend the analysis to strains associated with $BD$ functions of bounded deformation, which naturally arise in plasticity theory. \\

\emph{A brief literature review on the pairing.}  It is well known that the classical product rule \eqref{bbbb}
holds for smooth scalar functions $u$ and vector fields $\A$ in $\mathbb{R}^N$.
A first extension of this identity to the nonsmooth setting of
$BV$ functions and vector fields whose divergence belongs to a
Lebesgue space was obtained by Anzellotti \cite{Anz},
who introduced a pairing in such irregular setting. 
This construction provides a weak formulation of the product rule
and leads to generalized Gauss--Green formulas.
Later, in \cite{ChenFrid} the authors further developed this theory and proved that this pairing extends the classical scalar product
$\A\cdot\nabla u$.
Using a regularization argument, they proved that for every bounded
$u\in BV(\mathbb{R}^N)$ and every bounded divergence-measure
vector field $\A$ there exists a finite Radon measure
$(\A,Du)$ satisfying
\begin{equation}\label{eq:Leibniz}
\operatorname{div}(u\A)=u^*\,\operatorname{div}\A
+(\A,Du)
\end{equation}
in the sense of measures, where $u^*$ denotes the precise
representative of $u$. 
Subsequently, under the weaker assumption that $u^*\in L^1_{|\Div\A|}(\R^N)$ all the main properties and features 
as coarea, Leibniz and Gauss--Green formulas are proved in \cite{CD3}.

While the theory is well understood in the scalar–vector setting, the extension to tensor fields is not straightforward. The tensorial framework presents additional difficulties, since the natural pairing will involve the matrix contraction $\MA:D\bfu$ between the tensor field $\MA$ and the matrix field $D\bfu$, thus requiring a suitable extension of the divergence-measure theory to matrix-valued fields. \\

\emph{Main results.} 
Our aim is to introduce a notion of pairing between tensor fields with divergence measure and vector-valued functions of bounded variation, in order to extend the previous theory to the tensorial setting and, in particular, to get a tensor form of the Leibniz rule \eqref{eq:Leibniz}. 
The main result of the paper is Theorem \ref{main} below, which extends the validity of the Gauss–Green formula \eqref{eq:regularform}. It provides a generalization, under very weak assumptions and in a unified statement, of the classical divergence theorem in tensor form (see Theorem \ref{diverg} and its Corollary \ref{gaussclassica}) and of some partial extensions already present in the literature.
Specifically, the Gauss--Green formula considered in \cite[Remark 10.2.2]{ABM} is established for $\bfu\in BV_{\rm loc}(\R^N;\R^N)$ and regular tensor fields $\MA\in C^1(\R^N;\mathbb{R}^{N\times N})$. On the other hand, the case of $\bfu\in C^1(\R^N;\R^N)$ and irregular tensor fields is studied in \cite{DGMM}. 

The class of sets of finite perimeter provides a natural framework for the foundations of continuum mechanics, since it admits well-defined notions of oriented boundary and outer normal (see \cite{ChTo2,Zie}). Within the class of sets of finite perimeter, the following unifying result generalizes all the previous results for essentially bounded tensor fields.

\begin{theorem}\label{main}
 Let $\MA\in\DMAloc(\R^N;\mathbb{R}^{N\times N})$ and $\bfu\in BV_{\rm loc}(\R^N;\R^N)$ such that  $\bfu^*\in L^1_{|\DIV\MA|,{\rm loc}}(\R^N;\R^N)$. Let $E\subset \R^N$ be a bounded set of finite perimeter. Then the following Gauss--Green formulas hold
\begin{eqnarray}
& \displaystyle \int_{E^1} \bfu^* \cdot d\DIV\MA+ \int_{E^1}(\MA:D\bfu)  =- \int_{\partial^*E} \bfu^+\cdot \Trpv{\MA}{\partial^*E}\, d\Haus{N-1}\,, \label{eq:gaussgreen1} \\
& \displaystyle \int_{E^1\cup \partial^*E} \bfu^* \cdot d\DIV\MA + \int_{E^1\cup\partial^*E}(\MA:D\bfu)  =- \int_{\partial^*E} \bfu^-\cdot \Trmv{\MA}{\partial^*E}\, d\Haus{N-1} \,. \label{eq:gaussgreen2}
\end{eqnarray}
\end{theorem}
Here $E^1$ and $\partial^*E$ are the measure theoretic interior and reduced boundary of $E$ respectively, and 
$\Trace{\MA}{\partial^* E}$
are the normal traces of \(\MA\) when \(\partial^* E\) is oriented
with respect to the interior unit normal vector. Moreover,  $\bfu^*$ and $\bfu^\pm$ are the precise representative of $\bfu$ and the traces of $\bfu$ on $\partial^* E$ respectively, in the sense of $BV$ functions (see Section \ref{sec:preliminaries}) 
and $L^1_{|\DIV\MA|,{\rm loc}}(\R^N;\R^N)$ denotes the space of the locally integrable functions with respect to the measure $|\DIV\MA|$.

Theorem \ref{main} extends to the tensorial setting the Gauss--Green formulas obtained in \cite[Theorem 5.3]{Cas} for Lipschitz domains and in \cite[Theorem 5.1]{CD3} for sets of finite perimeter. 

The pairing measure $(\MA:D\bfu)$ appearing in formulas \eqref{eq:gaussgreen1} and \eqref{eq:gaussgreen2} will be introduced in Section \ref{sec:newpairing} in distributional sense, as follows: for every test function $\varphi\in C^\infty_c(\Omega)$
\begin{equation*}\label{eq:pairingvettoriale23}
\pscal{(\MA:D\bfu)}{\varphi} :=
-\int_\Omega \varphi \bfu\cdot  \DIV \MA \,dx -\int_\Omega \MA:[\bfu \otimes\nabla\varphi]\,dx.
\end{equation*}
When $\bfu$ is a regular function, $(\MA:D\bfu)$ coincides with the usual scalar product $\MA:D\bfu$ between the matrices $\MA$ and $D\bfu$. 
This pairing preserves the main
structural features of the classical theory, such as the
absolute continuity with respect to $|D\bfu|$ and Leibniz rule.
Finally, we specialize Theorem \ref{main} to some particular domains already considered in literature. More precisely, we establish some tensor Gauss–Green formulas on \emph{rough domains} of finite perimeter (see \eqref{eq:roughdom} below) and on the subclass of \emph{weakly regular} sets (satisfying \eqref{eq:weaklyregintro}). They will be proved independently in Theorem \ref{t:GGrough} and Theorem \ref{t:GGw}, respectively.
\begin{theorem}
\label{t:GGroughintro}
Let
$\MA\in\DMAloc(\R^N;\mathbb{R}^{N\times N})$ and $\bfu\in BV_{\rm loc}(\R^N;\R^N)\cap L^\infty_{\rm loc}(\R^N;\R^N)$. 
Let $E\subset\R^N$ be a bounded open set of finite perimeter satisfying 
\begin{equation}
\Haus{N-1}(\partial E \setminus E^0) < +\infty\,, 
\label{eq:roughdom}
\end{equation}
where $E^0$ denotes the measure-theoretic exterior of $E$. 
Then there exists $g\in L^\infty(\partial E\setminus E^0;\Haus{N-1})$ such that the following Gauss--Green formula holds
\begin{equation}\label{eq:gaussgreen1333intro} 
\displaystyle \int_{E} \bfu^* \cdot d\DIV\MA + \int_{E}(\MA:D\bfu) 
=\displaystyle - \int_{\partial^*E} \bfu^+\cdot \Trpv{\MA}{\partial^*E}\ d\Haus{N-1}
- \int_{\partial E\cap E^1} g(x)\ d\Haus{N-1}\,.\nonumber
\end{equation}
In particular, if
\begin{equation}
\Haus{N-1}(\partial E) = \Haus{N-1}(\partial^* E)\,,
\label{eq:weaklyregintro}
\end{equation}
the previous Gauss--Green formula reduces to
\begin{eqnarray*}
& \displaystyle \int_{ E} \bfu^* \cdot d\DIV\MA + \int_{ E}(\MA:D\bfu)  =- \int_{\partial E} \bfu^+\cdot \Trpv{\MA}{\partial E} d\Haus{N-1}\,.\label{eq:gaussgreen145intro} 
\end{eqnarray*}
Here $ \Trpv{\MA}{\partial E}$ and $ \Trpv{\MA}{\partial^* E}$
are the normal traces of \(\MA\) on \(\partial  E\) and \(\partial^*  E\), respectively,  when they are oriented
with respect to the interior unit normal vector.
\end{theorem}

\section{Preliminaries} \label{sec:preliminaries}

\medskip
\noindent\emph{Vectors and matrices.} 
Let $N \ge 2$. For ${\bf a}, {\bf b} \in \mathbb{R}^N$, we write ${\bf a} \cdot {\bf b}$ for the Euclidean scalar product and denote by $|{\bf a}| = \sqrt{{\bf a} \cdot {\bf a}}$ the associated norm. We denote by $\mathbb{R}^{N \times N}$ the set of real $N \times N$ matrices. Given $\mathbf A, \mathbf B \in \mathbb{R}^{N \times N}$, their Frobenius scalar product is defined as
\[
\mathbf A : \mathbf B := \mathrm{tr}(\mathbf A^T \mathbf B),
\]
where $\mathbf A^T$ is the transpose of $\mathbf A$ and $\mathrm{tr}(\mathbf A)$ its trace. The associated Frobenius norm is
\[
|\mathbf A| := \sqrt{\mathbf A : \mathbf A}.
\]
For ${\bf a}, {\bf b} \in \mathbb{R}^N$, the tensor product is defined by
\[
{\bf a} \otimes {\bf b} := {\bf a} \, {\bf b}^T \in \mathbb{R}^{N \times N}.
\]
The following identities (see, e.g., \cite[Section 3]{Gurtin1973}) follow directly from the definitions:
\begin{equation}
\mathbf A : ({\bf v} \otimes {\bf w}) = {\bf v} \cdot (\mathbf A {\bf w}) = {\bf w} \cdot (\mathbf A^T {\bf v})\,,
\label{eq:ident1}
\end{equation}
for all ${\bf v}, {\bf w} \in \mathbb{R}^N$ and $\mathbf A \in \mathbb{R}^{N \times N}$.

\medskip

Given a matrix $\mathbf A \in \mathbb{R}^{N \times N}$, we denote by $\DIV \mathbf A$ its distributional divergence, defined columnwise by
\[
\DIV \mathbf A := (\Div \mathbf A_1, \dots, \Div \mathbf A_N),
\]
where $\mathbf A_i$ is the $i$-th column of $\mathbf A$ and
\[
\Div \mathbf A_i := \sum_{j=1}^N \frac{\partial A_{ij}}{\partial x_j}, \quad i=1,\dots,N.
\]
Similarly, one may define the rowwise divergence
\[
\underline{\DIV}\, \mathbf A := (\Div \mathbf A^1, \dots, \Div \mathbf A^N),
\]
where $\mathbf A^i$ is the $i$-th row of $\mathbf A$ and
\[
\Div \mathbf A^i := \sum_{j=1}^N \frac{\partial A_{ji}}{\partial x_j}, \quad i=1,\dots,N.
\]
If $\mathbf A$ is symmetric, then $\underline{\DIV}\, \mathbf A = \DIV \mathbf A$.

Finally, observe that for every Borel set $B \subset \mathbb{R}^N$ one has the total variation estimate
\begin{equation}
|\DIV \mathbf A|(B) \le \sum_{i=1}^N |\Div \mathbf A_i|(B),
\label{eq:totvariationineq}
\end{equation}
see, e.g., \cite[p.~4]{AFP}.

\medskip

\noindent\emph{Measures.} The following definitions and basic facts about measures can be found, e.g., in \cite[Chapter 1]{AFP}.

We denote by
$\Leb{N}$ 
and $\Haus{N-1}$
the Lebesgue measure 
and the $(N-1)$-dimensional 
Hausdorff measure in $\R^N$, respectively. {Unless otherwise stated, a measurable set is understood to be $\Leb{N}$-measurable. 

Following the notation of \cite{AFP}, for every open set $\Omega\subseteq \R^N$ we denote by $\mathcal{M}_{\rm loc}(\Omega;\R^N)$ the space of vector Radon measures on $\Omega$, and by $\mathcal{M}(\Omega;\R^N)$ the space of finite vector Radon measures on $\Omega$, while by $\mathcal{M}(\Omega)$ and $\mathcal{M}_{\rm loc}(\Omega)$ the spaces of scalar Radon measures and scalar finite Radon measures on $\Omega$, respectively.

Given $\bm\mu\in\mathcal{M}_{\rm loc}(\Omega;\R^N)$ and a $\bm\mu$-measurable set $E$, the {\em restriction} 
$\bm\mu\res E$
is the Radon measure defined by
\[
\bm\mu\res E(B)=\bm\mu(E\cap B), \qquad \forall\ B\ \text{$\bm\mu$-measurable},\ 
B\subset\Omega.
\]

The {\em total variation} $|\bm\mu|$ of $\bm\mu \in \mathcal{M}_{\rm loc}(\Omega;\R^N)$ is 
the nonnegative Radon measure defined by
\[
|\bm\mu|(E) := \sup\left\{ \sum_{h=0}^\infty |\bm\mu(E_h)| \colon \ E_h\ 
\text{$\bm\mu$-measurable sets, pairwise disjoint},\ E=\bigcup_{h=0}^\infty E_h 
\right\},
\]
for every $\bm\mu$-measurable set $E \Subset \Omega$. If $\bm\mu \in \mathcal{M}(\Omega;\R^N)$, then $|\bm\mu|(\Omega) < \infty$.

\medskip  

A measure $\bm\mu \in \mathcal{M}_{\rm loc}(\Omega;\R^N)$ is {\em absolutely continuous} with respect to a given
nonnegative measure $\bm\nu$ (notation: $\bm\mu \ll \bm\nu$) 
if $|\bm\mu|(B)=0$ for every Borel set $B$ such that $\bm\nu(B)=0$. 
Two positive measures $\bm\nu_1, \bm\nu_2\in \mathcal{M}_{\rm loc}(\Omega;\R^N)$ are {\em mutually 
singular}
(notation: $\bm\nu_1 \perp \bm\nu_2$) 
if there exists a Borel set $E$ such that $\bm\nu_1|(E)=0$ and 
$|\bm\nu_2|(\Omega\setminus E) = 0$. Given $\bm\mu\in \mathcal{M}_{\rm loc}(\Omega;\R^N)$, the \emph{Lebesgue decomposition} of $\bm\mu$ with respect to $(\Leb{N})^N$ is 
\begin{equation*}
\bm\mu=\bm\mu^a + \bm\mu^s\,,
\end{equation*}
where $\bm\mu^a:=(\mu_i^a)_{i=1,\dots,N}$ is the {\em absolutely continuous part}, satisfying $\mu_i^a \ll \Leb{N}$, and $\bm\mu^s:=(\mu_i^s)_{i=1,\dots,N}$ is the {\em singular part}, satisfying $\mu_i^s \perp \Leb{N}$.

A measure $\bm\mu$ in $\Omega$ is \emph{concentrated on} $E\subset\Omega$ if $\bm\mu(\Omega\backslash E)=0$. The intersection of the closed sets $E\subset\Omega$ such that $\bm\mu$ is concentrated on $E$ is called the \emph{support of $\bm\mu$} and is denoted by ${\rm supp}(\bm\mu)$. In particular,
\begin{equation*}
\Omega\backslash {\rm supp}(\bm\mu) = \left\{x\in\Omega:\,\, \bm\mu(B_r(x))=0 \mbox{ for some $r>0$} \right\}\,.
\end{equation*}

We denote by $L^1_\mu(\Omega;\mathbb{R}^N)$ the space of $\mathbb{R}^N$-valued functions that are integrable with respect to $\mu$ on $\Omega$, and by $L^1_{\mu,\mathrm{loc}}(\Omega;\mathbb{R}^N)$ the space of functions that belong to $L^1_\mu(K;\mathbb{R}^N)$ for every compact set $K\subset\Omega$. In the scalar-valued case, we omit the target space and simply write $L^1(\Omega,\mu)$ and $L^1_{\mathrm{loc}}(\Omega,\mu)$.

\medskip

\medskip
\noindent\emph{Approximate limits.}
%\label{ss:def}
The following basic definitions and results can be found, e.g., in \cite[Sections 3.6 and 4.5]{AFP}.

We say that a function \(\bfu\in L^1_{\rm loc}(\Omega;\R^N)\) has an {\em approximate limit} 
\({\bf z}\in\R^N\) at
$x\in\Omega$ if
\begin{equation*}
\lim_{r\rightarrow0^{+}}\frac{1}{\Leb{N}\left(  B_r(x)\right)}\int_{B_r\left(  
x\right)
}\left|  \bfu(y)  - {\bf z}  \right|  \,dy=0\,;
\label{eq:approxlim1}
\end{equation*}
in this case we say that $x$ is a {\em Lebesgue point} of $\bfu$.
The set $S_{\bfu}\subset\Omega$ of points where this property does not hold is called the
{\em approximate discontinuity set} of $\bfu$, and, thanks to Lebesgue's differentiation theorem, we know that $\Leb{N}(S_\bfu) = 0$.
For any $x\in \Omega \setminus S_\bfu$ the approximate limit ${\bf z}$ is uniquely 
determined and is denoted by ${\bf z}=:\tilde{\bfu}(x)$. 

Given $\bfu \in L^1_{\rm loc}(\Omega;\R^N)$, we say that \(x\in\Omega\) is an {\em approximate jump point} of \(\bfu\) if
there exist \({\bf a}, {\bf b}\in\R^N\), \({\bf a}\neq {\bf b}\), and a unit vector \(\bm\nu\in\R^N\) such that 
\begin{equation}\label{f:disc}
\begin{gathered}
\lim_{r \to 0^+} \frac{1}{\Leb{N}(B_r^+(x))}
\int_{B_r^+(x)} \|\bfu(y) - {\bf a}\|\, dy = 0,
\\
\lim_{r \to 0^+} \frac{1}{\Leb{N}(B_r^-(x))}
\int_{B_r^-(x)} \|\bfu(y) - {\bf b}\|\, dy = 0,
\end{gathered}
\end{equation}
where \(B_r^+(x) := \{y\in B_r(x):\ (y-x)\cdot \bm\nu > 0\}\), and 
\(B_r^-(x) := \{y\in B_r(x):\ (y-x)\cdot \bm\nu < 0\}\).
The triplet \(({\bf a}, {\bf b}, \bm\nu)\), uniquely determined by \eqref{f:disc} 
up to a permutation
of \(({\bf a}, {\bf b})\) and a change of sign of \(\bm\nu\),
is denoted by \((\bfu^+(x), \bfu^-(x), \bm\nu_u(x))\).
The set of approximate jump points of \(\bfu\) is denoted by \(J_{\bfu}\), and it clearly satisfies $J_{\bfu} \subset S_{\bfu}$. It is easy to check that $J_\bfu$ and $S_\bfu$ are Borel sets, and $\widetilde{\bfu}$, $\bfu^+$ and $\bfu^-$ are Borel functions. 

Finally, for $\bfu \in L^1_{\rm loc}(\Omega;\R^N)$ we define the {\em precise representative} of $\bfu$ in $x \in \Omega$ as
\begin{equation*} \label{def:precise_repr}
\bfu^{*}(x) := \lim_{r\to0^+}\frac{1}{\Leb{N}\left(  B_r(x)\right)} \int_{B_r(x)}\bfu(y) \, d y,
\end{equation*}
whenever the limit exists. It is then clear that 
\begin{equation}
\bfu^*(x)=
\begin{cases}
\tilde{\bfu}(x) & x\in \Omega \setminus S_{\bfu}, \\
\displaystyle \frac{\bfu^+(x)+ \bfu^-(x)}{2} & x\in J_{\bfu}.
\end{cases}
\label{eq:preciserepresentative}
\end{equation} 
A priori, it is not clear whether $\bfu^*$ is well posed in $S_{\bfu} \setminus J_{\bfu}$, in general. However, for sufficiently regular functions it is known that $S_{\bfu} \setminus J_{\bfu}$ is suitably small (for instance, when $\bfu$ is a function of bounded variation, see Section \ref{sec:BV}). 

Let $\bfu\in L^1(\Omega;\R^N)$, and consider the sequence 
\begin{equation}
\bfu_\varepsilon:= \eta_\varepsilon * \bfu = (\eta_\varepsilon * u_1, \eta_\varepsilon * u_2, \dots, \eta_\varepsilon * u_N)\,, 
\label{eq:mollification}
\end{equation}
where $\eta_\varepsilon(x):= \frac{1}{\varepsilon^N} \eta(\frac{x}{\varepsilon})$ and $\eta$ is a positive symmetric mollifier. Then, combining \cite[Proposition 3.64(b) and Proposition 3.69(b)]{AFP} we get
\begin{equation}
\bfu_\varepsilon \to \bfu^* \quad \mbox{ pointwise in $\Omega\backslash(S_\bfu\backslash J_\bfu)$ as $\varepsilon \to 0$.}
\label{eq:convtoprecise}
\end{equation}

\medskip
\subsection{Functions of bounded variation and sets of finite perimeter. } \label{sec:BV}
For a detailed treatment of the theory of $BV$ functions, we refer the reader to the monograph \cite{AFP}.

We say that \(\bfu\in L^1(\Omega)\) is a function of bounded variation in \(\Omega\)
if its distributional derivative \(D\bfu\in \mathbb{R}^{N\times N}\) %of \(\bfu\) 
is a finite Radon matrix-valued measure in \(\Omega\).
The vector space of all functions of bounded variation in \(\Omega\)
is denoted by \(BV(\Omega)\).

If \(\bfu\in BV(\Omega)\), then \(D\bfu\) can be decomposed as
the sum of the absolutely continuous and the singular part with respect
to the Lebesgue measure, i.e.\
\[
D\bfu = D^a \bfu + D^s \bfu,
\qquad D^a \bfu = \nabla \bfu \, \Leb{N},
\]
where \(\nabla \bfu\) is the approximate gradient of \(\bfu\),
defined \(\Leb{N}\)-a.e.\ in \(\Omega\)
(see \cite[Section~3.9]{AFP}).
On the other hand, 
the jump set $J_{\bfu}$ is countably $\Haus{N-1}$--rectifiable,
$\Haus{N-1}(S_\bfu  \setminus J_\bfu) = 0$
(see \cite[Definition~2.57 and Theorem~3.78]{AFP}), and
the singular part \(D^s \bfu\) can be further decomposed
as the sum of its Cantor and jump part, i.e.
\[
D^s \bfu = D^c \bfu + D^j \bfu,
\qquad
D^c \bfu := D^s \bfu \res (\Omega\setminus S_\bfu),
\quad
D^j \bfu := D^s \bfu \res J_\bfu,
\]
where the symbol \(\bm\mu\res B\) denotes the restriction of the measure \(\bm\mu\)
to the set \(B\).
We will denote by \(D^d \bfu := D^a \bfu + D^c \bfu\) the diffuse part of the measure \(D\bfu\).

The precise representative $\bfu^*$ of $\bfu$ is defined in
$\Omega\setminus(S_\bfu\setminus J_\bfu)$
(hence $\Haus{N-1}$-a.e.\ in $\Omega$) as in \eqref{eq:preciserepresentative}. The mollified functions $\bfu_\varepsilon$ pointwise converge to
$\bfu^*$ in its domain
(see \cite[Corollary~3.80]{AFP}).

Let $\Omega$ be an open set with Lipschitz boundary. Then for $\Haus{N-1}$-a.e. $x\in\partial \Omega$, there exists a vector $\bfu_{\Omega}(x)\in\R$ such that
\begin{equation}
\lim_{\rho\to0}\rho^{-N}\int_{\Omega\cap B_\rho(x)} |\bfu(y)-\bfu_{\Omega}(x)|\,\mathrm{d}y=0
\label{eq:boundarytracethm}
\end{equation} 
(see \cite[Theorem 3.87]{AFP}). The function $\bfu_\Omega\in L^1(\partial\Omega)$ is called a \emph{trace} of $\bfu$ on $\partial\Omega$. 

We conclude this section recalling an approximation result by smooth functions, due to Anzellotti and Giaquinta (see, e.g., \cite[Theorem 3.9]{AFP} combined with \cite[Remarks 1.18 and 2.12]{GiustiMinSur}).

\begin{theorem}\label{thm:AG}
Let $\bfu\in BV(\Omega;\R^N)$. Then there exists a sequence $(\bfu_k)_{k\in\N}\subset C^\infty(\Omega;\R^N)$ such that 
\begin{equation*}
\bfu_k\to \bfu \quad \mbox{ in }L^1(\Omega;\R^N) \quad \mbox{ and } \quad \lim_{k\to+\infty}\int_\Omega |\nabla \bfu_k|\,\mathrm{d}x =  \int_\Omega|D\bfu|\,.
\end{equation*}
If, in addition, $\Omega$ is a bounded open set with Lipschitz boundary, then $\bfu_{k,\Omega}=\bfu_\Omega$ for all $k$. 
\end{theorem}

For every $M>0$, let
\begin{equation*}\label{gtrun}
T_M(s) := \max\{\min\{s, M\}, -M\}
\,,
\qquad s\in\R.
\end{equation*}
For $\bfu \in \BorelN$, define the componentwise truncation map $\bm\tau_M: \R^N \to \R^N$ by
\begin{equation*}
\bm\tau_M(\bfu):=(T_M(u_1), T_M(u_2), \dots, T_M(u_N))\,,
\end{equation*}
and, for brevity, set
\begin{equation}
\bfu^{M}:= \bm\tau_M(\bfu)\,.
\label{eq:truncfunct}
\end{equation}
Since $T_M$ is 1-Lipschitz and $T_M(0)=0$, it follows that $\bm\tau_M$ is 1-Lipschitz, $\bm\tau_M({\bf 0})={\bf 0}$ and
\begin{equation}
\|\bm\tau_M(\bfu)\| \leq  \|\bfu\|\,.
\label{eq:contraction}
\end{equation}

\begin{proposition}[Properties of the truncated functions] 
\label{p:trunc}
Let $\bfu\in BV(\Omega; \R^N)$, and for every $M>0$, define $\bfu^{M}$ as in \eqref{eq:truncfunct}. Then $\bfu^{M} \in BV(\Omega; \R^N)$ and the following properties hold:
\begin{enumerate}
\item[(i)]
$(\bfu^\pm)^{M} = (\bfu^{M})^\pm \to \bfu^\pm$
$\Haus{N-1}$-a.e.\ in $\Omega$. In particular, $(\bfu^{M})^*\to \bfu^*$,
$\Haus{N-1}$-a.e.\ in $\Omega$;
\item[(ii)]
$|D \bfu^{M}| \leq |D\bfu|$ in the sense of measures, for every $M>0$;
\item[(iii)]
$\|[\bfu^{M}]^\pm\| \leq \|\bfu^\pm\|$ for every $M>0$, hence
\[
\|({\bfu^*})^{M}\| \leq \frac{1}{2}(\|\bfu^+\|+\|\bfu^-\|)
\qquad
\forall M>0 \quad \mbox{ for $\Haus{N-1}$-a.e. $x\in\Omega$;}
\]
\item[(iv)] if $\mu\in \mathcal M(\Omega)$ with $\mu \ll \Haus{N-1}$ and $\bfu\in BV(\Omega;\R^N)\cap L^1_{\mu}(\Omega;\R^N)$, then
$({\bfu^*})^{M} \to {\bfu^*}$ in $L^1_{\mu}(\Omega;\R^N)$.
\end{enumerate}
\label{prop:truncconv}
\end{proposition} 
\begin{proof}
The fact that $\bfu^M \in BV(\Omega;\R^N)$ for every $M$ and the identities in (i) follow from the general chain rule in \cite[Theorem 2.1]{AmbrosioDalMaso} (and the argument of its proof), using that $\bm\tau_M({\bf 0})={\bf 0}$ and $\bm\tau_M$ is 1-Lipschitz, which also implies (ii). From \cite[Proposition 3.69(c)]{AFP}, if $x\in J_\bfu$, then $x\in J_{\bfu^M}$ if and only if $(\bfu^+)^M(x)\neq (\bfu^-)^M(x)$, in which case $(\bfu^M)^\pm(x) =(\bfu^\pm)^M(x)$. Otherwise, $x\not\in S_{\bfu^M}$ and $\widetilde{\bfu^M}(x) = \widetilde{\bfu}^M(x)=(\bfu^-)^M(x) = (\bfu^+)^M(x)$. Therefore, we can write
\begin{equation*}
(\bfu^M)^\pm(x) =(\bfu^\pm)^M(x) \quad \mbox{ for all $x\in\Omega\backslash(S_\bfu\backslash J_\bfu)$,}
\end{equation*}
hence $\Haus{N-1}$-a.e. in $\Omega$. Note also that for $M> \max\{\|\bfu^+(x)\|, \|\bfu^-(x)\|\}$ we have $(\bfu^M)^\pm(x) = \bfu^\pm(x)$, and this implies that $[\bfu^{M}]^\pm \to \bfu^\pm$ $\Haus{N-1}$-a.e. in $\Omega$. 
Property (iii) follows from \eqref{eq:contraction}, the identities in (i) and the definition of $\bfu^*$. Finally, (iv) is a consequence of (i), the bound in (iii), and the Lebesgue Dominated Convergence Theorem. 
\end{proof}

Let \(E\subset\R^N\) be an \(\Leb{N}\)-measurable set.
For every open set \(\Omega\subset\R^N\) the \emph{perimeter} \(P(E, \Omega)\)
is defined by
\[
P(E, \Omega) := \sup\left\{
\int_E \Div \bm\varphi\, dx:\ \bm\varphi\in C^1_c(\Omega, \R^N),\ \|\bm\varphi\|_\infty\leq 1
\right\}.
\]
We say that \(E\) is of \emph{finite perimeter} in \(\Omega\) if \(P(E, \Omega) < +\infty\).
Denoting by \(\chi_E\) the characteristic function of \(E\),
if \(E\) is a set of finite perimeter in \(\Omega\), then
\(D\chi_E\) is a finite Radon measure in \(\Omega\) and
\(P(E,\Omega) = |D\chi_E|(\Omega)\).
If \(\Omega\subset\R^N\) is the largest open set such that \(E\)
is locally of finite perimeter in \(\Omega\),
we call \emph{reduced boundary} \(\partial^* E\) of \(E\) the set of all points
\(x\in \Omega\) in the support of \(|D\chi_E|\) such that the limit
\[
\nuint_E(x) := \lim_{\rho\to 0^+} \frac{D\chi_E(B_\rho(x))}{|D\chi_E|(B_\rho(x))}
\]
exists in \(\R^N\) and satisfies \(|\nuint_E(x)| = 1\).
The function \(\nuint\colon\partial^* E\to S^{N-1}\) is called
the measure theoretic unit interior normal to \(E\).
We recall the classical result by De Giorgi (see, e.g., \cite[Theorem~3.59]{AFP}) which states that
\(\partial^* E\) is countably \((N-1)\)-rectifiable
and
\begin{equation} 
|D\chi_E| = \Haus{N-1}\res \partial^* E\,.
\label{eq:varchi}
\end{equation}
Given \(E\subset\R^N\) an \(\Leb{N}\)-measurable set, 
for every \(t\in [0,1]\) we denote by \(E^t\) the set
\[
E^t := \left\{x\in\R^N:\
\lim_{\rho\to 0^+} \frac{\LLN(E\cap B_\rho(x))}{\LLN(B_\rho(x))} = t\right\}
\]
of all points where \(E\) has density \(t\).
The sets \(E^0\), \(E^1\), \(\partial^e E := \R^N\setminus (E^0 \cup E^1)\) are called 
respectively the measure theoretic exterior, the measure theoretic interior and
the essential boundary of \(E\).
If \(E\) has finite perimeter in \(\Omega\), Federer's structure theorem
states that
\(\partial^* E\cap\Omega \subset E^{1/2} \subset \partial^e E\)
and \(\Haus{N-1}(\Omega\setminus(E^0\cup \partial^e E \cup E^1)) = 0\)
(see \cite[Theorem~3.61]{AFP}).

\subsection{The pairing for vector fields}\label{pimo}

Let us consider
\begin{equation*}
\DMloc(\Omega;\R^N) :=\left\{ \A\in L^\infty_{\rm{loc}}(\Omega; \R^N):\,\, \Div\A\in \mathcal M(\Omega)\right\},
\end{equation*}
that is, the space of all locally bounded
vector fields whose distributional divergence is a bounded Radon measure in \(\Omega\).

Then, according to \cite{Anz,ChenFrid}, the pairing between $\A\in\DMloc(\Omega;\R^N)$ and $D u$, for a given function $u \in BV_{\rm loc}(\Omega) \cap L^{\infty}_{\rm loc}(\Omega)$ is the distribution given by
\begin{equation}\label{f:pairing}
 \varphi \in C^{\infty}_{c}(\Omega) \to \ban{(\A, Du), \varphi} := - \int_{\Omega}\varphi u^* \, d \Div\A - \int_{\Omega} u \A \cdot \nabla \varphi \, dx,
 \end{equation}
where $u^{*}$ is the precise representative of the $BV$-function $u$. Since the measure $\Div \A$ is absolutely continuous with respect to $\Haus{N-1}$, and $u^*$ is well defined $\Haus{N-1}$-almost everywhere, this definition is well-posed.
As proven in  \cite{ChenFrid}, it holds that
\begin{equation}
u \A\in\DMloc(\Omega;\R^N)\,,
\label{eq:uAmeas}
\end{equation}
and the pairing is a Radon measure in \(\Omega\), absolutely continuous with respect to \(|Du|\), satisfying the following Leibniz-type formula:
\begin{equation} 
(\A, Du) = \Div (u \A) - u^{*} \Div \A. 
\label{f:anz}
\end{equation} 
In \cite{CD3} the authors, redefining the pairing distribution 
for functions $u \in BV_{\rm loc}(\Omega)$ such that $u^* \in L^1_{|\Div \A|,{\rm loc}}(\Omega)$ (thus requiring a dependence between the scalar functions and the vector field), 
still obtained a measure with the same properties as before, and proved a coarea formula and a Leibniz rule, and finally achieved generalizations of the Gauss--Green formulas.
The following representation result can be found in \cite[Theorem 4.12]{CD3} and \cite[Theorem 3.5]{CCDM}.

\begin{theorem}\label{t:pairingCD3}
Let \(\A\in\DMloc(\Omega;\R^N)\) and \(u\in BV_{\rm loc}(\Omega)\cap L^1_{|\Div \A|,{\rm loc}}(\Omega)\).
Then the measure \((\A, Du)\) admits the following decomposition:
\begin{itemize}
	\item[(i)]
	absolutely continuous part: 
	\((\A, Du)^a = \A \cdot \nabla u\, \Leb{N}\);
		
	\item[(ii)]
	jump part:
	\(\displaystyle
	(\A, Du)^j = 
	\frac{\Trp{\A}{J_u}+\Trm{\A}{J_u}}{2}
	\, (u^+-u^-) \, \Haus{N-1} \res J_u
	\);
	
	\item[(iii)]
	Cantor part:
	\begin{equation*}\label{eq:cantorpart}
	(\A, Du)^c\res (\Omega\setminus S_{\A})= \widetilde{\A} \cdot D^c u\res (\Omega\setminus S_{\A}),
	\end{equation*}
	where \(S_{\A}\) is the approximate discontinuity set  of \(\A\).

\end{itemize}
\end{theorem}

\subsection{Matrix-valued fields with measure divergence}

We consider the space of essentially bounded matrix fields with measure divergence:
\begin{equation*}
\DMA(\Omega;\mathbb{R}^{N\times N}) :=\left\{ \MA \in L^\infty(\Omega;\mathbb{R}^{N\times N}) : \DIV\MA\in \mathcal M(\Omega;\R^N)\right\}.
\end{equation*}

The following lemma provides an approximation result for stress fields. See \cite[Theorem 1.2]{ChenFrid} for the analogous result concerning vector fields, and \cite[Lemma 2.3]{KT} for more regular stress fields. 

\begin{lemma}
Given $\MA\in \DMA(\Omega;\mathbb{R}^{N\times N})$, there exists $\{\MA_k\}\subset C^\infty({\Omega};\mathbb{R}^{N\times N})$ such that 
\begin{enumerate}
\item[(i)] $\MA_k \to \MA$ in $L^1(\Omega;\mathbb{R}^{N\times N})$;
\item[(ii)] $\displaystyle \lim_{k\to+\infty} \int_\Omega |\DIV\MA_k| =  |\DIV\MA|(\Omega)$.
\end{enumerate}
\label{lem:approxim}
\end{lemma}

\begin{proof}
The existence of the sequence $\{\MA_k\}$ satisfying $(i)$ and $(ii)$ can be obtained by adapting the arguments of \cite[Theorems 1.1 and 1.2]{ChenFrid} to the matrix-valued setting. 
\end{proof}

The next result shows that the divergence measure of $\MA$ is absolutely continuous with respect to $\Haus{N-1}$. 

\begin{proposition}\label{ABS}
Let $\MA \in \DMA(\Omega;\mathbb{R}^{N\times N})$, and let $B\subset\Omega$ be a Borel measurable set such that $\Haus{N-1}(B)=0$. Then $|\DIV \MA|(B)=0$. 
\label{prop:absolutecont}
\end{proposition}

\begin{proof}
In view of \eqref{eq:totvariationineq}, it will suffice to apply \cite[Proposition 3.1]{ChenFrid} to each real-valued measure $|\Div \MA_i|$, $i=1,\dots,N$. 
\end{proof}

As a consequence, the set
\begin{equation}\label{f:jump}
\jump{\MA} 
:= \left\{
x\in\Omega:\
\limsup_{r \to 0^+}
\frac{|\DIV \MA| (B_r(x))}{r^{N-1}} > 0
\right\},
\end{equation} 
is a Borel set, \(\sigma\)-finite with respect to \(\Haus{N-1}\). Moreover,
the measure \(\DIV \MA\) admits the decomposition
\[
\DIV\MA = \DIV^a\MA + \DIV^c\MA + \DIV^j\MA,
\]
where \(\DIV^a\MA\) is absolutely continuous with respect to \(\Leb{N}\),
\(\DIV^c\MA (B) = 0\) for every Borel set \(B\) with \(\Haus{N-1}(B) < +\infty\),
and
\begin{equation}\label{divj}
\DIV^j\MA = {\bf f}\, \Haus{N-1}\res\jump{\MA}
\end{equation}
for some Borel vector-valued function \({\bf f}\)
(see \cite[Proposition~2.5]{ADM}).

\subsection{Normal traces}
\label{distrtraces}
We first recall the definition of traces of the normal component of a vector field \(\C\in \DM(\Omega;\R)\), which can be defined as distributions
\(\Trace{\C}{\Sigma}\)
on every oriented countably \(\Haus{N-1}\)--rectifiable set
\(\Sigma\subset\Omega\)
(see \cite{AmbCriMan,Anz,ChenFrid}).
According to \cite{AmbCriMan}, given a domain \(\Omega'\Subset\Omega\) of class \(C^1\), the trace of the normal component of \(\C\) on \(\partial\Omega'\) is defined in the distributional sense by
\begin{equation}\label{f:disttr}
\pscal{\Trace[]{\C}{\partial\Omega'}}{\varphi}
:= \int_{\Omega'} \C\cdot \nabla\varphi\, dx + \int_{\Omega'} \varphi\, d\Div\C,
\qquad
\forall\varphi\in C^\infty_c(\Omega).
\end{equation}
It is shown in \cite[Proposition 3.2]{AmbCriMan} that this distribution is in fact induced by an \(L^\infty\) function on \(\partial\Omega'\),
which we continue to denote by \(\Trace[]{\C}{\partial\Omega'}\), and satisfies the estimate
\[
\|\Trace[]{\C}{\partial\Omega'}\|_{L^\infty(\partial\Omega')}
\leq \|\C\|_{L^\infty(\Omega';\R)}.
\]
Now, since \(\Sigma\) is oriented and countably $\Haus{N-1}$--rectifiable,
we can find countably many {oriented} \(C^1\) hypersurfaces \(\Sigma_i\),
with classical normal vectors \(\bm\nu_{\Sigma_i}\),
and pairwise disjoint Borel sets \(E_i\subseteq \Sigma_i\)
such that \(\Haus{N-1}(\Sigma\setminus \bigcup_i E_i) = 0\).
Moreover, we may assume, without loss of generality, that for each $i$ there exist two open, bounded sets \(\Omega_i, \Omega'_i\) with \(C^1\) boundaries
and exterior normal vectors \(\bm\nu_{\Omega_i}\) and \(\bm\nu_{\Omega_i'}\) respectively,
such that
\(E_i\subseteq \partial\Omega_i \cap \partial\Omega'_i\)
and
\[
\bm\nu_{\Sigma_i}(x) = \bm\nu_{\Omega_i}(x) = -\bm\nu_{\Omega'_i}(x)
\qquad \forall x\in E_i.
\]
We then define the orientation on \(\Sigma\) by setting
\(\bm\nu_{\Sigma}(x) := \bm\nu_{\Sigma_i}(x)\) for
\(\Haus{N-1}\)-a.e.\ $x\in E_i$.
Using the localization property established in \cite[Proposition 3.2]{AmbCriMan},
the normal traces of \(\C\) on \(\Sigma\) are defined by
\[
\Trm{\C}{\Sigma} := \Tr(\C, \partial\Omega_i),
\quad
\Trp{\C}{\Sigma} := -\Tr(\C, \partial\Omega'_i),
\qquad
\Haus{N-1}-\text{a.e.\ on}\ E_i.
\]
These two normal traces belong to
\(L^{\infty}(\Sigma, \Haus{N-1}\res\Sigma)\) (see \cite[Proposition 3.2]{AmbCriMan})
and satisfy
\begin{equation}\label{mmm}
\Div \C \res\Sigma =
\left[\Trp{\C}{\Sigma} - \Trm{\C}{\Sigma}\right]
\, \Haus{N-1} \res\Sigma\,.
\end{equation}

For our purposes, we need to extend definition \eqref{f:disttr} to a matrix-valued field. 
Given a domain \(\Omega'\) as above, we define
the trace of the normal component of \(\MA\in\DMA(\Omega;\mathbb{R}^{N\times N})\) on \(\partial\Omega'\) in the distributional sense as the $n$-tuple of scalar distributions
\begin{equation*}
\Tracev[]{\MA}{\partial\Omega'}:= \biggl(\Trace[]{\MA_1}{\partial\Omega'}, \Trace[]{\MA_2}{\partial\Omega'},\dots, \Trace[]{\MA_N}{\partial\Omega'}\biggr)\,.
\end{equation*}
Then, recalling \eqref{f:disttr}, we have
\begin{equation}\label{f:disttrv}
\begin{split}
\pscal{\Tracev[]{\MA}{\partial\Omega'}}{\bm \varphi}
& = \sum_{i=1}^N \pscal{\Trace[]{\MA_i}{\partial\Omega'}}{\varphi_i} \\
&=\int_{\Omega'} \MA:\nabla\bm\varphi \, dx + \int_{\Omega'} \bm\varphi \cdot d\DIV\MA,
\qquad
\forall \bm\varphi\in C^\infty_c(\Omega;\R^N).
\end{split}
\end{equation}
Therefore, this distribution is induced by a vector-valued \(L^\infty\) function on \(\partial\Omega'\) (which we still
denote by \(\Tracev[]{\MA}{\partial\Omega'}\)), the estimate
\[
\|\Tracev[]{\MA}{\partial\Omega'}\|_{L^\infty(\partial\Omega';\mathbb{R}^{N})}
\leq \|\MA\|_{L^\infty(\Omega';\mathbb{R}^{N\times N})}
\label{eq:normtraceestimate}
\]
holds, together with the localization property
\begin{equation*}
\Tracev[]{\MA}{\partial\Omega_1} = \Tracev[]{\MA}{\partial\Omega_2} \quad \Haus{N-1}-\text{a.e.\ on}\ E,
\end{equation*}
if $E$ is a Borel subset of $\partial\Omega_1 \cap \partial\Omega_2$ and $\bm\nu_{\Omega_1}=\bm\nu_{\Omega_2}$ on $E$. This allows us to define the normal traces of \(\A\) on \(\Sigma\) by
\[
\Trmv{\MA}{\Sigma} := \Trv(\MA, \partial\Omega_i),
\quad
\Trpv{\MA}{\Sigma} := -\Trv(\MA, \partial\Omega'_i),
\qquad
\Haus{N-1}-\text{a.e.\ on}\ E_i\,,
\]
thus extending \cite[Definition 3.3]{AmbCriMan} to a matrix-valued field. 
These two normal traces belong to
\(L^{\infty}_{{\small \Haus{N-1}\res\Sigma}}(\Sigma; \R^N)\) 
and, arguing as for \cite[Proposition 3.4(ii)]{AmbCriMan} with minor changes, we have 
\begin{equation*}\label{mmmv}
\DIV^j \MA \res\Sigma =
\left[\Trpv{\MA}{\Sigma} - \Trmv{\MA}{\Sigma}\right]
\, \Haus{N-1} \res\Sigma\,.
\end{equation*}
On the other hand, it holds that
\begin{equation*}
\Trp{\MA}{\Sigma}\not=\Trm{\MA}{\Sigma}
\qquad
\text{$\Haus{N-1}$-a.e. in $\Sigma\cap \jump{\MA}$}.
\end{equation*}
By \eqref{divj}, it then follows that for every Borel set $B\subset \Omega$ 
\begin{equation*}\label{Divj2}
|\DIV^j \MA\res\Sigma|(B)=0 \ \ \ \Rightarrow\ \ \ \Haus{N-1}\res \Theta_{\MA}(B)=0,
\end{equation*}
that is,
\begin{equation*}\label{Divj3}
\Haus{N-1}\res \Theta_{\MA} \ll |\DIV^j \MA|.
\end{equation*}

\section{A new pairing for tensor fields} \label{sec:newpairing}

Let us introduce a new pairing between a matrix-valued field $\MA$ and the matrix derivative of a function in $BV(\Omega;\R^N)$. In analogy with \cite{CD3}, for $\MA \in {\rm DM}^\infty(\Omega;\mathbb{R}^{N\times N})$, we introduce the function space
\begin{equation*}
BV(\Omega;\R^N) \cap L^1_{|\DIV\MA|}(\Omega;\R^N) := \left\{\bfu \in BV(\Omega;\R^N):\,\, \bfu^* \in  L^1_{|\DIV\MA|}(\Omega;\R^N)\right\} \,. 
\end{equation*}

Then, for every \(\MA\in\DMA(\Omega;\mathbb{R}^{N\times N})\) and $\bfu\in BV(\Omega;\R^N) \cap L^1_{|\DIV\MA|}(\Omega;\R^N)$ we define the linear functional
\((\MA: D\bfu) \colon C^\infty_c(\Omega) \to \R\) by
\begin{equation}\label{eq:pairingvettoriale}
\pscal{(\MA:D\bfu)}{\varphi} :=
-\int_\Omega \varphi \bfu^*\cdot d \DIV\MA - \int_\Omega  \MA:[\bfu \otimes\nabla\varphi]\, dx. 
\end{equation}
It is easy to check that
\begin{equation}
(\MA:D\bfu)=\sum_{i=1}^N(\A_i, Du_i)\,,
\label{eq:pairingsum}
\end{equation}
where $(\A_i,Du_i)$ is the pairing \eqref{f:pairing}. 
Similarly, we can define
\begin{equation*}\label{eq:pairingvettoriale2}
\pscal{\underline{(\MA:D\bfu)}}{\varphi} :=
-\int_\Omega \varphi \bfu^*\cdot d \underline{\DIV}\MA - \int_\Omega  \, \MA:[\bfu \otimes\nabla\varphi]\, dx. 
\end{equation*}
and we have
$$
\underline{(\MA:D\bfu)}=\sum_{i=1}^N(\A^i, Du_i).
$$
Moreover, $\underline{(\MA:D\bfu)}=(\MA:D\bfu)$ whenever $\MA$ is symmetric.
In the following we focus on $(\MA:D\bfu)$; the same arguments and results can be applied to $\underline{(\MA:D\bfu)}$.

\begin{remark}
From the classical product rule \eqref{bbbb} it is immediate to check that, if $\MA\in C^1(\Omega;\mathbb{R}^{N\times N})$ and $\bfu\in C^1(\Omega;\R^N)$, then   
\begin{equation*}\label{bbbb1}
(\MA:D\bfu)=\MA:\nabla\bfu\,\Leb{N}
\end{equation*}
in the sense of measures.
\end{remark}

The structural properties of the space $BV(\Omega;\mathbb{R}^N)$, in particular the fact that each component of a vector-valued $BV$ function still belongs to $BV$, are crucial in order to deduce the following result from \cite[Theorem 4.12]{CD3}, in view of property \eqref{eq:pairingsum}. However, we prefer to provide an independent proof for the reader’s convenience and for future reference.

\begin{theorem}\label{thm:main}
Let $\MA\in\DMA(\Omega;\mathbb{R}^{N\times N})$ and $\bfu\in BV(\Omega;\R^N) \cap L^1_{|\DIV\MA|}(\Omega;\R^N)$. Then, the distribution \((\MA: D\bfu)\) defined in \eqref{eq:pairingvettoriale} is a Radon measure in \(\Omega\), absolutely continuous with respect to \(|D\bfu|\), and it holds that for every open set $U\Subset\Omega$
\begin{equation}
|(\MA: D\bfu)|(B) \leq \|\MA\|_{L^\infty(U)} |D\bfu|(B) \quad \mbox{for every Borel set $B \subset U$.}
\label{eq:abscont}
\end{equation}
Moreover, $\Div(\MA^T \bfu)\in \mathcal{M}(\Omega;\R^N)$ and the equation
\begin{equation}\label{f:anz1}
\Div(\MA^T \bfu)= \bfu^* \cdot \DIV\MA + (\MA: D\bfu)
\end{equation}
holds in the sense of measures in \(\Omega\). 
\end{theorem}
\begin{proof}
{We first assume that $\bfu\in BV(\Omega;\R^N)\cap L^\infty(\Omega;\R^N)$. In this case, }we may follow the ideas of \cite[Theorem 3.1]{ChenFrid}. Let $\{\MA_k\}$ be the sequence of Lemma~\ref{lem:approxim}, and let $\{\bfu_k\}$ be the sequence approximating $\bfu$ given in Theorem~\ref{thm:AG}. Then, arguing as in \cite{ChenFrid}, we can prove that 
\begin{equation}
\begin{split}
\int_\Omega |\Div(\MA_k^T\bfu_k)|\, dx & = \sup \left  \{\int_\Omega (\MA_k^T\bfu_k) \cdot \nabla \varphi \, dx: \,\, \varphi\in C^1_c(\Omega)\,,\,\, \|\varphi\|_\infty\leq 1 \right \} \\
& \leq 3 \left \{\|\bfu\|_\infty|\DIV \MA_k|(\Omega)+ \|\MA\|_\infty\|\nabla \bfu_k\|_{L^1(\Omega)} \right\}
\end{split}
\label{eq:equiboundeddiv}
\end{equation}
whence, passing to the limit as $k\to+\infty$, we get
\begin{equation*}
|\Div(\MA^T \bfu)|(\Omega) \leq3 \left \{\|\bfu\|_\infty|\DIV \MA|(\Omega)+ \|\MA\|_\infty|D \bfu|(\Omega) \right\}
<+\infty\,.
\label{eq:equiboundeddiv111}
\end{equation*} 
Since $\MA^T \bfu \in L^\infty(\Omega;\mathbb{R}^N)$, we conclude that $\MA^T \bfu \in \mathcal{DM}^\infty(\Omega;\mathbb{R}^N)$. 

We subdivide into steps the rest of the proof. \\
\noindent
{\bf Step 1.} First, we assume that $\bfu$ is Lipschitz continuous on compact subsets of $\Omega$. We claim that \eqref{f:anz1} holds in the form
\begin{equation}
\Div(\MA^T \bfu)= \bfu \cdot \DIV\MA + \MA: \nabla\bfu\,\Leb{N}
\label{eq:leibniz1}
\end{equation}
in the sense of measures on $\Omega$. 

From Lemma \ref{lem:approxim} we have $\DIV\MA_k \rightharpoonup \DIV\MA$ as Radon measures on $\Omega$. Then
\begin{equation*}
\bfu \cdot \DIV\MA_k +\MA_k: \nabla\bfu\,\Leb{N} \rightharpoonup \bfu \cdot \DIV\MA + \MA: \nabla\bfu\,\Leb{N} \quad \mbox{ in $\mathcal{M}(\Omega)$.}
\end{equation*}
Moreover, $\Div(\MA_k^T \bfu)\rightharpoonup\Div(\MA^T \bfu)$ in the sense of distributions. Collecting these facts we can compute the limit in the sense of distributions of the identity
\begin{equation*}
\Div(\MA_k^T \bfu)= \bfu \cdot \DIV\MA_k + \MA_k: \nabla\bfu\,\Leb{N}\,,
\end{equation*} 
so that \eqref{eq:leibniz1} holds in the sense of distributions. Since $C^1_c(\Omega)$ is dense in $C_c(\Omega)$, the same identity holds in the sense of measures. \\
\\
\noindent
{\bf Step 2.} Let $\bfu\in BV(\Omega;\R^N)\cap L^\infty(\Omega;\R^N)$, and define the sequence of mollifications $\{\bfu_\varepsilon\}$ as in \eqref{eq:mollification}. Then, from {\bf Step 1} we get 
\begin{equation*}
\Div(\MA^T\bfu_\varepsilon)= \bfu_\varepsilon \cdot \DIV\MA + \MA: \nabla\bfu_\varepsilon\,\Leb{N} \quad \mbox{ in $\mathcal{M}(\Omega)$.}    
\end{equation*}
Since $\Haus{N-1}(S_\bfu\backslash J_\bfu)=0$, from \eqref{eq:convtoprecise} we have that $\bfu_\varepsilon$ converges to $\bfu^*$ $\Haus{N-1}$-a.e. $x$ in $\Omega$ and, by virtue of Proposition \ref{prop:absolutecont}, also $|\DIV\MA|$-a.e. on $\Omega$. Then, by 
the Dominated Convergence Theorem applied to the measure $\DIV\MA$, we obtain
\begin{equation}
\bfu_\varepsilon \cdot \DIV\MA \rightharpoonup \bfu^* \cdot \DIV\MA \,\, \mbox{ as $\varepsilon\to0$  in $\mathcal{M}(\Omega)$.}
\label{eq:stima1}
\end{equation}  
Arguing as for \eqref{eq:equiboundeddiv}, we get that the sequence $\{\Div(\MA^T\bfu_\varepsilon)\}$ is equibounded in $\mathcal{M}(\Omega)$. Since $\Div(\MA^T\bfu_\varepsilon)$ converges to $\Div(\MA^T\bfu)$ on $\Omega$ in the sense of distributions, we then have
\begin{equation}
\Div(\MA^T\bfu_\varepsilon) \rightharpoonup \Div(\MA^T\bfu) \,\, \mbox{ as $\varepsilon\to0$  in $\mathcal{M}(\Omega)$.}
\label{eq:stima2}
\end{equation}
Collecting \eqref{eq:stima1} and \eqref{eq:stima2}, we obtain
 \begin{equation*}
\MA: \nabla\bfu_\varepsilon \rightharpoonup \Div(\MA^T\bfu) - \bfu^* \cdot \DIV\MA = (\MA: D\bfu) \,\, \mbox{ as $\varepsilon\to0$  in $\mathcal{M}(\Omega)$,}
\label{eq:stima3}
\end{equation*}
so that $(\MA: D\bfu)$ is a Radon measure on $\Omega$. To prove \eqref{eq:abscont}, we can follow the argument of \cite[Theorem 1.5]{Anz}. 
 \\
\noindent
{\bf Step 3.} Now, we deal with the general case $\bfu\in BV(\Omega;\R^N) \cap L^1_{|\DIV\MA|}(\Omega;\R^N)$, for which we can adapt the second part of the argument for the proof of \cite[Proposition 4.4]{CDCM}. Let $M>0$ and $\bfu^M$ be defined as in \eqref{eq:truncfunct}. Then, by Proposition \ref{prop:truncconv} we have $\bfu^M\in BV(\Omega;\R^N)\cap L^\infty(\Omega;\R^N)$, $(\bfu^M)^* \to \bfu^*$ $\Haus{N-1}$- a.e. in $\Omega$ and in $L^1_{|\DIV\MA|}(\Omega;\R^N)$. 

Passing to the limit as $M\to+\infty$ in
\begin{equation*}
\pscal{(\MA:D\bfu^M)}{\varphi} =
-\int_\Omega \varphi (\bfu^M)^*\cdot d \DIV \MA - \int_\Omega  \, \MA:[\bfu^M \otimes\nabla\varphi]\, dx\,,
\end{equation*}
we obtain $(\MA:D\bfu)$ as the weak* limit of $(\MA:D\bfu^M)$ in the sense of measures. Therefore, $(\MA: D\bfu)$ is a Radon measure on $\Omega$. Moreover, using Proposition~\ref{prop:truncconv}(ii) and \eqref{eq:abscont} for $(\MA: D\bfu^M)$, for each $B\subset U\Subset \Omega$ we get 
\begin{equation*}
|(\MA: D\bfu^M)|(B) \leq \|\MA\|_{L^\infty(U)} |D\bfu^M|(B) \leq \|\MA\|_{L^\infty(U)} |D\bfu|(\Omega) \quad \mbox{ for every $M$,}
\end{equation*}
whence \eqref{eq:abscont} follows also for $(\MA: D\bfu)$. Finally, $\Div(\bfu\MA)$ is a measure and \eqref{f:anz1} holds. 
\end{proof}

\begin{remark}\label{pipo}
If $\DIV \MA<\!\!<\Leb{N}$ (and, with a slight abuse of notation, we denote its density also by $\DIV \MA$), the pairing can be defined equivalently as
\begin{equation*}\label{eq:pairingvettoriale2}
\pscal{(\MA:D\bfu)}{\varphi} :=
-\int_\Omega \varphi \bfu\cdot  \DIV \MA \,dx -\int_\Omega  \, \MA:[\bfu \otimes\nabla\varphi]\, dx.
\end{equation*}
It is well posed under the assumption that $\DIV \MA\in L^N(\Omega;\R^N)$, since $\bfu\in BV(\Omega;\R^N)\subset L^{\frac N{N-1}}(\Omega;\R^N)$ {(see, e.g., \cite[Theorem 3.47]{AFP})}, if $N>1$.  In this case, H\"{o}lder's inequality yields $\bfu\in L^1_{|\DIV\MA|}(\Omega;\R^N)$ and
\begin{equation*}
|\Div(\MA^T\bfu)|(\Omega) \leq3 \left \{\|\bfu\|_{\frac N{N-1}}\|\DIV \MA\|_N+ \|\MA\|_\infty|D \bfu|(\Omega) \right\}
<+\infty\,,
\label{eq:equiboundeddiv1117}
\end{equation*} 
whence $\Div(\MA^T\bfu)\in \mathcal{M}(\Omega)$. 
For $N=1$ we have $\bfu\in BV(\Omega)\subset L^{\infty}_{\rm{loc}}(\Omega)$. Then $\bfu\in L^1_{|D\A|}(\Omega)$ and
\begin{equation*}
|D(\MA^T\bfu)|(\Omega) \leq3 \left \{\|\bfu\|_{\infty}\|\DIV \MA\|_1+ \|\MA\|_\infty|D \bfu|(\Omega) \right\}
<+\infty\,,
\label{eq:equiboundeddiv11176}
\end{equation*} 
so that $D(\MA^T\bfu)\in \mathcal{M}(\Omega)$. 
\end{remark}

If $\bfu \in BV(\Omega;\R^N)\cap L^\infty(\Omega;\R^N)$, as a consequence of Theorem~\ref{thm:main} we have \(\MA^T\bfu\in\DM(\Omega;\R^N)\). Then, according to definition \eqref{f:disttr}, we can consider the normal traces of $\MA^T\bfu$ on a countably \(\Haus{N-1}\)--rectifiable set $\Sigma\subset\Omega$. The following result provides a representation formula for these traces, in the same vein of \cite[Proposition 3.1]{CD3}.
\begin{proposition}\label{p:tracesb}
	Let \(\MA\in\DMA(\Omega;\mathbb{R}^{N\times N})\),  \({\bfu}\in BV(\Omega;\R^N)\cap L^\infty(\Omega;\R^N)\) and let \(\Sigma\subset\Omega\) 
	be a countably \(\Haus{N-1}\)--rectifiable set, oriented
	as in Section~\ref{distrtraces}.
	Then the normal traces of \(\MA^T\bfu\) on \(\Sigma\) are given by
	\begin{equation}\label{f:trus}
		\Trace{\MA^T\bfu}{\Sigma} = 
		\begin{cases}
			\bfu^\pm\cdot \Trpmv{\MA}{\Sigma},
			& \Haus{N-1}-\text{a.e.\ in}\ J_{\bfu}\cap\Sigma,\\
			\widetilde{\bfu}\cdot\Trpmv{\MA}{\Sigma},
			& \Haus{N-1}-\text{a.e.\ in}\ \Sigma\setminus J_{\bfu}. 	
		\end{cases}
	\end{equation}
\end{proposition}
\proof
It will suffice to show
\eqref{f:trus} for \(\Tr^-\), since the proof for \(\Tr^+\) is analogous. Recalling the notation of Section~\ref{distrtraces}, we may assume that $J_\bfu$ is oriented with $\bm\nu_\Sigma$ on $J_\bfu\cap\Sigma$. We adapt the argument of Proposition 3.1 in \cite{CD3}.

We will perform a blow-up argument around each point \(x\in \Sigma\) such that
\begin{itemize}
	\item[(a)] \(x\in\Omega\setminus (S_{\bfu}\setminus J_{\bfu})\), 
	\(x\in E_i\) for some \(i\), the set \(E_i\) has density \(1\) at \(x\), 
	and \(x\) is a Lebesgue point of both \(\Trmv{\MA}{\Sigma}\) and \(\Trm{\MA^T\bfu}{\Sigma}\)
	with respect to
	\(\Haus{N-1}\res\partial\Omega_i\);
	\\
	\item[(b)]
	\(|\DIV\MA| \res\Omega_i (B_\varepsilon(x)) = o(\varepsilon^{N-1})\)  \quad \mbox{ and }\quad \(|\Div (\MA^T\bfu)| \res\Omega_i  (B_\varepsilon(x)) = o(\varepsilon^{N-1})\)
	as $\varepsilon\to0$.
	\end{itemize}
Note that \(\Haus{N-1}\)-a.e.\ \(x\in \Sigma\) satisfies (a), since $\Haus{N-1}(S_\bfu\backslash J_\bfu)=0$, and (b), as a consequence of \cite[Theorem~2.56 and (2.41)]{AFP}. 
In order not to overburden the notation, we set
 \(\bfu^-(x) := \widetilde{\bfu}(x)\) if \(x\in \Omega\setminus S_{\bfu}\).

Let \(\eta\in C^{\infty}_c(\R^N)\) be a cut-off function, with ${\rm supp}\,\eta\subseteq B_1(0)$, such that \(0\leq \eta \leq 1\), and for every \(\varepsilon > 0\) set \(\eta_{\varepsilon}(y) := \eta\left(\frac{y-x}{\varepsilon}\right)\). We will use both the definitions of trace \eqref{f:disttr} and \eqref{f:disttrv}. In particular, for $\varepsilon$ small enough, we can choose $\eta_\varepsilon$ and $\bm\varphi=\bfu^-(x)\eta_{\varepsilon}$  as test functions in \eqref{f:disttr} and \eqref{f:disttrv}, respectively. Note that
$$
\nabla\bm\varphi(y)=\bfu^-(x)\otimes\nabla\eta_{\varepsilon}(y),
$$
hence by \eqref{eq:ident1}
$$
\MA:\nabla\bm\varphi(y)=\MA:(\bfu^-(x)\otimes\nabla\eta_{\varepsilon}(y))=\bfu^-(x)\cdot(\MA\nabla\eta_{\varepsilon}(y)) = \nabla\eta_{\varepsilon}(y) \cdot (\MA^T \bfu^-(x))\,.
$$
Then, for \(\varepsilon > 0\) small enough, we can write
\begin{equation}\label{f:tra}
\begin{split}
&\frac{1}{\varepsilon^{N-1}}
\int_{\partial\Omega_i} 
[\Tr(\MA^T \bfu, \partial\Omega_i) -
\bfu^-(x) \cdot\Trv(\MA, \partial\Omega_i)]
\, \eta_{\varepsilon}(y)\, d\Haus{N-1}(y)
\\ = {} &%%%%%
\frac{1}{\varepsilon^{N-1}}
\int_{\Omega_i} \nabla\eta_{\varepsilon}(y) \cdot \MA^T(y)\bfu(y)\, dy+\frac{1}{\varepsilon^{N-1}}
\int_{\Omega_i} \eta_{\varepsilon}(y) \, d\Div(\MA^T\bfu)
\\ & -\frac{1}{\varepsilon^{N-1}} 
\int_{\Omega_i} \bfu^-(x)\cdot(\MA(y) \nabla\eta_{\varepsilon}(y) )\, dy
-\frac{1}{\varepsilon^{N-1}} 
\int_{\Omega_i} \eta_{\varepsilon}(y)\bfu^-(x)\cdot d\DIV \MA(y)\,
\\ = {} &
\frac{1}{\varepsilon^{N-1}}
\int_{\Omega_i} \nabla\eta_{\varepsilon}(y) \cdot [\MA^T(y)\bfu(y) - \MA^T(y)\bfu^-(x)]\, dy
\\ & + 
\frac{1}{\varepsilon^{N-1}}
\int_{\Omega_i} \eta_{\varepsilon}(y) \, d[\Div(\MA^T\bfu) - \bfu^-(x)\cdot \DIV\MA](y)\, \\
=: & \,\,\,\, J_1(\varepsilon) + J_2(\varepsilon) \,.
\end{split}
\end{equation}

Using the change of variables \(z = (y-x)/\varepsilon\),
as \(\varepsilon \to 0\) the left-hand side of this equality converges to
\[
[\Trm{\MA^T\bfu}{\Sigma}(x) - \bfu^-(x)\cdot\Trmv{\MA}{\Sigma}] \int_{\Pi_x} \eta(z)\, d\Haus{N-1}(z)\,,
\]
where \(\Pi_x\) is the tangent plane to \(\Sigma_i\) at \(x\).
Clearly \(\eta\) can be chosen in such a way that
\(\int_{\Pi_x} \eta\, d\Haus{N-1} > 0\).

Therefore, in order to prove \eqref{f:trus} for \(\Tr^-\) it suffices
to show that the two integrals \(J_1(\varepsilon)\) and \(J_2(\varepsilon)\)
on the right hand side of \eqref{f:tra} converge to \(0\)
as \(\varepsilon \to 0\).

With the change of variables \( z = (y-x) / \varepsilon\) 
we have that
\[
J_1(\varepsilon) =
\int_{\Omega_i^\varepsilon} \MA^T(x+\eps z)[ \bfu(x+\varepsilon z) - \bfu^-(x)] \cdot \nabla\eta(z) 
\, dz,
\]
where
\[
\Omega_i^\varepsilon := \frac{\Omega_i - x}{\varepsilon}.
\]
As \(\varepsilon\to 0\), these sets locally converge to the half-space 
\(P_x := \{z\in\R^N:\ \pscal{z}{ \bm\nu(x)} < 0\}\),
hence
\[
\lim_{\varepsilon\to 0}
\int_{\Omega_i^\varepsilon\cap B_1} |\bfu(x+\varepsilon z) - \bfu^-(x)|\, dz  = 
\lim_{\varepsilon\to 0}
\int_{P_x \cap B_1} |\bfu(x+\varepsilon z) - \bfu^-(x)|\, dz  = 0
\]
(see \cite[Remark 3.85]{AFP}) so that
\[
|J_1(\varepsilon)| \leq
\|\MA\|_{L^\infty(B_\eps(x))}\, \|\nabla\eta\|_{\infty} 
\int_{\Omega_i^\varepsilon\cap B_1} |\bfu(x+\varepsilon z) - \bfu^-(x)|\, dz  
\to 0.
\]
As for $J_2(\varepsilon)$, using (b) we have
\begin{equation*}
\limsup_{\varepsilon\to0}|J_2(\varepsilon)| \leq \limsup_{\varepsilon\to 0} |\bfu^-(x)| \frac{|\DIV \MA| (B_\varepsilon(x))}{\varepsilon^{N-1}} + \limsup_{\varepsilon\to 0} \frac{|\Div(\MA^T\bfu)| (B_\varepsilon(x))}{\varepsilon^{N-1}} = 0\,,
\end{equation*}
so that $J_2(\varepsilon)\to0$ and the proof is complete. 
\endproof

\begin{remark}
Let $\Omega'\Subset\Omega$ be an open set with Lipschitz boundary. Then \eqref{f:trus} holds with $\Sigma=\partial\Omega'$ and $\bfu^+(x)= \bfu^{\Omega'}(x)$, $\bfu^-(x)= \bfu^{\Omega\backslash\overline{\Omega'}}(x)$, where these traces are defined for $\Haus{N-1}$-a.e. $x\in\partial\Omega'$ by \eqref{eq:boundarytracethm}. See \cite[Lemma 5.6]{Cas} for the analogous result in the scalar case. 
\end{remark}

Since $J_{\bfu}$ is countably $\Haus{N-1}$-rectifiable,
Proposition~\ref{p:tracesb} immediately yields the following result, which is the analog of \cite[Corollary 3.2]{CD3}.

\begin{corollary}\label{c:traces}
	Let \(\MA\in\DMA(\Omega;\mathbb{R}^{N\times N})\) and \(\bfu\in BV(\Omega;\R^N)\cap L^\infty(\Omega;\R^N)\). Then the normal traces of \(\MA^T\bfu\) on \(J_{\bfu}\) are given by
	\begin{equation*}\label{f:tr}
		\Trace{\MA^T\bfu}{J_{\bfu}} = \bfu^\pm \cdot \Tracev{\MA}{J_\bfu},
		\qquad \Haus{N-1}-\text{a.e.\ in}\ J_\bfu. 	
	\end{equation*}
	In particular
	\begin{equation}\label{f:truA}
		\Div(\MA^T\bfu)\res J_\bfu =
		\left[
		\bfu^+ \cdot \Trpv{\MA}{J_\bfu} - \bfu^- \cdot \Trmv{\MA}{J_\bfu}
		\right]\, \Haus{N-1} \res J_\bfu.
	\end{equation}
\end{corollary}

We are now ready to state the main decomposition theorem
for the pairing measure. 

\begin{theorem}\label{t:pairing} 
Let \(\MA\in\DMA(\Omega;\mathbb{R}^{N\times N})\) and $\bfu \in BV(\Omega;\R^N) \cap L^1_{|\DIV\MA|}(\Omega;\R^N)$. 
Then the measure \((\MA : D\bfu)\) admits the following decomposition:
\begin{itemize}
	\item[(i)]
	absolutely continuous part: 
	\((\MA : D\bfu)^a = \MA: \nabla \bfu\, \Leb{N}\);
	
	\item[(ii)]
	jump part:
	\(\displaystyle
	(\MA : D\bfu)^j = \Trstarv{\MA}{J_\bfu}
		\cdot (\bfu^+-\bfu^-) \, \Haus{N-1} \res J_\bfu
	\).

\item[(iii)]
	Cantor part: \((\MA:D\bfu)^c\res(\Omega\setminus S_{\MA}) = 
	\widetilde{\MA} : D^c \bfu \res(\Omega\setminus S_{\MA})\), where \(S_{\MA}\) is the approximate discontinuity set  of \(\MA\).
\end{itemize}
\end{theorem}
\proof
The proof of (i) is immediate from \cite[Theorem 4.12]{CD3}, in view of \eqref{eq:pairingsum}, since $(\MA_i,Du_i)^a = \MA_i \cdot \nabla u_i\, \Leb{N}  $ for every $i=1,\dots, N$. Assertion (ii) follows by combining \eqref{f:anz1} and \eqref{f:truA}, once we note that $(\MA : D\bfu)^j$ is concentrated on $J_\bfu$. Regarding (iii), since $S_{\MA} = \bigcup_{i=1}^N S_{\MA_i}$, we have $\Omega \backslash S_{\MA} \subseteq \Omega\backslash S_{\MA_i}$ for every $i=1,\dots,N$. Therefore, by \cite[Theorem 4.12]{CD3},
\begin{equation*}
(\MA_i,Du_i)^c\res(\Omega\setminus S_{\MA}) = 
	\widetilde{\MA}_i \cdot D^c u_i \res(\Omega\setminus S_{\MA}) \quad \mbox{ for every $i=1,\dots,N$}
\end{equation*}
and the conclusion follows by \eqref{eq:pairingsum}. 
\endproof
\begin{remark}[\(BV\) matrix-valued fields]\label{mmmm}
Let \(\bfu\in BV(\Omega;\R^N)\) and \(\MA \in BV_{\rm loc}(\Omega;\mathbb{R}^{N\times N}) \cap L^\infty_{{\rm loc}}(\Omega;\mathbb{R}^{N\times N})\). 
Then clearly \(\MA \in \DMA(\Omega;\mathbb{R}^{N\times N})\) and
\[
\Tracev{\MA}{J_\bfu} = \MA^\pm_{J_{\bfu}}\,\bm\nu_{\bfu}\,, 
\qquad
\text{\(\Haus{N-1}\)-a.e.\ in}\ J_{\bfu},
\]
where \(\MA^\pm_{J_{\bfu}}\) are the traces of \(\MA\) on \(J_{\bfu}\)
(see \cite[Theorem~3.77]{AFP}).
Hence, the jump part of \((\MA:D{\bfu})\) can be written as
\[
(\MA:D{\bfu})^j = \frac{\MA^+ + \MA^-}{2}\, :D^j {\bfu}=\MA^*:D^j{\bfu}.
\]
In particular, when $\MA\in C^1(\Omega;\mathbb{R}^{N\times N})$ and \(\bfu\in BV(\Omega;\R^N)\) we have
\begin{equation*}\label{bbbb1}
(\MA:D\bfu)=\MA:D{\bfu},
\end{equation*}
in the sense of measures.
\end{remark}

\begin{remark}\label{r:ipoCantor}
The Cantor part of the pairing measure is completely represented as in Theorem~\ref{t:pairing} (iii) only if 
\begin{equation}\label{f:ipo}
	|D^c \bfu| (S_{\MA}) = 0\,.
	\end{equation}
This requirement may represent a significant restriction to the applicability of Theorem~\ref{t:pairing}.
Indeed, $|D^c \bfu|(S_{\MA})$ can be arbitrarily large: by adapting \cite[Example 3.9]{CCDM}, one can construct 
a field $\MA\in \DMA(\Omega;\mathbb{R}^{N\times N})$ whose set $S_{\MA}$ has Hausdorff dimension equal to $N$. 

On the other hand, we identify some relevant cases in which \eqref{f:ipo} is satisfied. First, we note that it is equivalent to \(|D^d \bfu|(S_{\MA}) = 0\), since \(\Leb{N}(S_{\MA}) = 0\).
This implies that it is satisfied, for instance, when \(S_{\MA}\) is 
$\sigma$--finite with respect to $\Haus{N-1}$
(see \cite[Proposition~3.92(c)]{AFP}), hence by any \(\MA\in BV_{{\rm loc}}(\Omega ; \mathbb{M}^{N\times  N})\cap L^\infty_{{\rm loc}}(\Omega ; \mathbb{M}^{N\times  N})\). Looking at the assumptions on $\bfu$, \eqref{f:ipo} holds when $D^c \bfu = {\bf 0}$,  
i.e.\ if 
\(\bfu \in SBV(\Omega;\R^N)\). Moreover, since the set $\jump{\MA}$ defined in \eqref{f:jump} is
$\sigma$-finite with respect to $\Haus{N-1}$, we remark that assumption \eqref{f:ipo} is equivalent to
$|D^c \bfu| (S_{\MA}\setminus \jump{\MA}) = 0$.
\end{remark}

\subsection{Proof of Gauss--Green formulas}
We are going to prove Theorem \ref{main}. The proof follows an argument similar to that of \cite[Theorem 5.1]{CD3}.

\begin{proof}[Proof of Theorem \ref{main}]
Since \(E\) is bounded, it is not restrictive to suppose that
\(\MA\in\DMA(\R^N;\mathbb{R}^{N\times N})\) and \(\bfu\in BV(\R^N;\R^{N})\) with $\bfu^*\in L^1_{|\DIV\MA|}(\R^N;\R^N)$.
We divide the proof into two steps.

\noindent
{\bf Step 1}. 
We first assume that $ \bfu\in BV(\R^N;\R^{N}) \cap L^\infty(\R^N;\R^N)$.
Recall that $\chi_E\in BV(\R^N)$, the reduced boundary $\partial ^*E$ is a $\Haus{N-1}$-rectifiable set and that
\begin{equation}\label{chichi}
\chi_E^*=\chi_{E^1}+\frac12\chi_{\partial ^*E}.
\end{equation}
Let us consider the vector field ${\B}:=\MA^T\bfu$. Then \(\chi_E {\B}\) has compact support,
hence
$
\Div(\chi_E {\B}) (\R^N) = 0
$
(see \cite[Lemma 3.1]{ComiPayne}). Moreover, since by \eqref{eq:uAmeas} 
the vector field ${\B}$ belongs to \(\DM(\R^N;\R^{N})\), by using the pairing for scalar $BV$ functions recalled in Section \ref{pimo}, and choosing in \eqref{f:anz} $\chi_E$ and ${\B}$ in place of $u$ and $\A$, respectively, we get
\begin{equation}\label{GreenIIIA}
\int_{\R^N} \chi_E^* \, d\Div {\B} =
- ( {\B} , D\chi_E) (\R^N).
\end{equation}

By  the definition of normal traces, formula \eqref{mmm} and Proposition \ref{p:tracesb} 
it holds
\begin{equation}\label{bubu}
\Div {\B}\res \partial^* E =
( \bfu^+ \cdot \Trpv{\MA}{\partial^*E} -  \bfu^- \cdot\Trmv{\MA}{\partial^*E}) \Haus{N-1} \res \partial^* E\,.
\end{equation}
With this, \eqref{chichi} and \eqref{bubu}, the left hand side of \eqref{GreenIIIA} can be rewritten as
\begin{equation}\label{f:gl}
\int_{\R^N} \chi_E^* \, d\Div {\B}  =
\int_{E^1} \, d\Div {\B} + \frac12\int_{\partial ^*E} \, [ \bfu^+\cdot\Trpv{\MA}{\partial^*E}- \bfu^-\cdot\Trmv{\MA}{\partial^*E}] \ d \Haus{N-1}\,.
\end{equation}

On the other hand, for what concerns the right hand side of \eqref{GreenIIIA}, by \cite[Remark 4.4]{CD3} and \eqref{eq:varchi}
we obtain
\[
( {\B},D\chi_E)= \frac{\Trp{{\B}}{\partial^*E} + \Trm{{\B}}{\partial^*E}}{2}
\, \Haus{N-1}\res \partial^* E.
\]
Recalling that $ \bfu^\pm\in L^1(\partial^* E, \Haus{N-1}\res\partial^*E)$ (see \cite[Theorem~3.84]{AFP}) and applying Proposition \ref{p:tracesb}, 
it follows that
\begin{equation}\label{f:gr}
( \B, D\chi_E) (\R^N) = 
\int_{\partial ^*E}\frac12 [ \bfu^+\cdot\Trpv{\MA}{\partial^*E}+ \bfu^-\cdot\Trmv{\MA}{\partial^*E}]\ d\Haus{N-1}\,.
\end{equation}
Finally, plugging \eqref{f:gl} and \eqref{f:gr} into \eqref{GreenIIIA} and simplifying some terms,
and recalling the definition of $\B$, we obtain
\begin{equation}\label{nono1}
\int_{E^1} \, d\Div(\MA^T\bfu) =-\int_{\partial ^*E}\bfu^+\cdot\Trpv{\MA}{\partial^*E}\ d\Haus{N-1}.
\end{equation}
Using now formula \eqref{f:anz1} and the new pairing $(\MA:D\bfu)$, the left hand side of \eqref{nono1} can be rewritten as
$$
\int_{E^1} \, d(\bfu^* \cdot \DIV\MA + (\MA: D\bfu)) =-\int_{\partial ^*E}\bfu^+\cdot\Trpv{\MA}{\partial^*E}\ d\Haus{N-1}.
$$
This concludes the proof of 
 \eqref{eq:gaussgreen1}.
In order to prove \eqref{eq:gaussgreen2} it suffices to note that, by \eqref{nono1}
\[
\begin{split}
&\int_{E^1\cup \partial ^*E} \, d\Div(\MA^T\bfu)=\int_{E^1} \, d\Div(\MA^T\bfu) +\int_{\partial ^*E} \, [ \bfu^+\cdot\Trpv{\MA}{\partial^*E}- \bfu^-\cdot\Trmv{\MA}{\partial^*E}] \ d\Haus{N-1}
\\
=&-\int_{\partial ^*E} \bfu^+\cdot\Trpv{\MA}{\partial^*E}\ d\Haus{N-1}+\int_{\partial ^*E} \, [ \bfu^+\cdot\Trpv{\MA}{\partial^*E}- \bfu^-\cdot\Trmv{\MA}{\partial^*E}] \ d\Haus{N-1}\,,
\end{split}
\]
hence \eqref{eq:gaussgreen2} follows.
This concludes the proof of  Step 1.\\
\\
\noindent
{\bf Step 2:} We now assume $\bfu\in BV_{\rm loc}(\R^N;\R^N)$ such that $\bfu^*\in L^1_{|\DIV\MA|,{\rm loc}}(\R^N;\R^N)$. We prove \eqref{eq:gaussgreen1} by means of a truncation argument, which can be compared with \cite[Theorem 5.1]{CD3}. 

For every $M>0$, we define the truncated function $\bfu^M\in BV_{\rm loc}(\R^N;\R^N)\cap L^\infty(\R^N;\R^N)$ as in \eqref{eq:truncfunct} so that, in view of the Step 1, we can write
\begin{equation*}
 \displaystyle \int_{E^1} (\bfu^M)^* \cdot d\DIV\MA + \int_{E^1}(\MA:D\bfu^M)  = - \int_{\partial ^* E} (\bfu^M)^{+}\cdot \Trpv{\MA}{\partial ^* E} d\Haus{N-1}\,.
\end{equation*}
By Proposition \ref{prop:truncconv}(iv), we have
\begin{equation*}
\lim_{M\to+\infty} \int_{E^1} (\bfu^M)^* \cdot d\DIV\MA = \int_{E^1}\bfu^* \cdot d\DIV\MA\,,
\end{equation*}
while arguing as in the Step 3 of the proof of Theorem \ref{thm:main}, we get
\begin{equation*}
\lim_{M\to+\infty} \int_{E^1}(\MA:D\bfu^M) = \int_{E^1}(\MA:D\bfu) \,.
\end{equation*}

To conclude, by Proposition \ref{prop:truncconv}(i) and (iii) we have $(\bfu^M)^+ \to \bfu^+$ in $L^1(\partial^* E; \Haus{N-1})$, so that we obtain
\begin{equation*}
\lim_{M\to+\infty} \int_{\partial^* E} (\bfu^M)^{+}\cdot \Trpv{\MA}{\partial^* E} d\Haus{N-1} =  \int_{\partial^* E} \bfu^{+}\cdot \Trpv{\MA}{\partial^* E} d\Haus{N-1}\,,
\end{equation*}
which combined with the previous estimates gives \eqref{eq:gaussgreen1}. 
\end{proof}

\section{Gauss--Green formula for weakly regular sets}

In this section we derive another version of
the Gauss--Green formula for a particular class of bounded open sets $E\subset\R^N$ with finite perimeter, namely the so-called {\it {weakly regular}} sets (see \cite[Definition 2.5]{LeoSa}), for which
\begin{equation}
\Haus{N-1}(\partial E) = \Haus{N-1}(\partial^*E)\,.
\label{eq:weaklyreg}
\end{equation}
In order to do that, we will need some technical results which may also be of independent interest. More precisely, we establish the following
extension theorems and gluing constructions involving vector-valued functions and tensor-valued fields (analogous results for vector fields have been obtained, for instance, in \cite{ChToZi, ComiPayne, CD4}).

We first consider extensions on overlapping domains. Here, the overlap is an open set containing the boundary $\partial E$ of a bounded set of finite perimeter in $\Omega$, and the gluing takes place along $\partial E$ by restricting the respective vector fields to $E$ and its complement.

\begin{theorem}
\label{extension}
Let $\Omega\subset\R^N$ be an open set, $E$ be a set of finite perimeter in $\Omega$ and $U,W\subset \R^N$ open sets such that
	$W \Subset {\rm int}(E) \subset E \Subset U \subset\Omega$.
	Let 
\(\MA_1\in\DMAloc(U;\mathbb{R}^{N\times N})\) and \(\MA_2\in\DMAloc(\Omega\setminus\overline{W};\mathbb{R}^{N\times N})\).
	Given $\bfu_1\in BV_{\rm{loc}}(U;\R^N)\cap L^\infty_{\rm{loc}}(U;\R^N)$ and $\bfu_2\in BV_{\rm{loc}}(\Omega\setminus\overline{W};\R^N)\cap L^\infty_{\rm{loc}}(\Omega\setminus\overline{W};\R^N)$,
	let ${\bf v}_i(x) := \MA_i^T(x)\bfu_i(x)$, $i = 1,2$. 
	Then  ${\bf v}_1\in \DMloc(U;\R^N)$, ${\bf v}_2\in \DMloc(\Omega\setminus\overline{W};\R^N)$, the extension
	\[
	{\bf v}(x) :=
	\begin{cases}
	{\bf v}_1(x),
	& \text{if}\ x\in E,\\
	{\bf v}_2(x),
	& \text{if}\ x\in\Omega\setminus E,
	\end{cases}
	\]
	belongs to $\DMloc(\Omega;\R^N)$ and
	\[
	\Div{\bf v}= \chi_{E^1} \Div{\bf v}_1 + \chi_{E^0} \Div{\bf v}_2 +
	[\Trp{{\bf v}_1}{\partial^* E} - \Trm{{\bf v}_2}{\partial^* E}]\,
	\Haus{N-1}\res\partial^*E.  
	\]
\end{theorem}
 \begin{proof}
In view of \eqref{eq:uAmeas}, it is sufficient to use \cite[Theorem 5.1]{ComiPayne} with $F={\bf v}$ and $F_i={\bf v}_i$, $i=1,2$. 
\end{proof}

We now address extensions on complementary domains. Here, the two vector fields are extended by zero on their complements and summed to obtain the gluing
along $\partial U$. 
\begin{theorem}
\label{t:gluing}
Let $U\Subset\Omega\subset\R^N$ be open sets with $\Haus{N-1}(\partial U) < \infty$.
Let 
\(\MA_1\in\DMAloc(U;\mathbb{R}^{N\times N})\) and \(\MA_2\in\DMAloc(\Omega\setminus\overline{U};\mathbb{R}^{N\times N})\). 
	Given $\bfu_1\in BV_{\rm{loc}}(U;\R^N)\cap L^\infty_{\rm{loc}}(U;\R^N)$ and 
$\bfu_2\in BV_{\rm{loc}}(\Omega\setminus\overline{U};\R^N)\cap L^\infty_{\rm{loc}}(\Omega\setminus\overline{U};\R^N)$,
define
\[
{\bf v}_1(x) := \begin{cases}
\MA_1^T(x)\bfu_1(x),
&\text{if}\ x\in U,\\
0,
&\text{if}\ x\in\Omega\setminus{U},
\end{cases}
\quad
{\bf v}_2(x) := \begin{cases}
0,
&\text{if}\ x\in \overline{U}\,,\\
\MA_2^T(x)\bfu_2(x),
&\text{if}\ x\in\Omega\setminus\overline{U}\,,
\end{cases}
\]
and set
$$
{\bf v}(x):={\bf v}_1(x)+{\bf v}_2(x).
$$
Then ${\bf v}_1\in \DMloc(U;\R^N)$, ${\bf v}_2\in \DMloc (\Omega\setminus\overline{U};\R^N)$,
${\bf v}\in \DMloc(\Omega;\R^N)$ and the divergence measure of the extension can be represented as
\[
\Div{\bf v}= \chi_{U^1} \Div{\bf v}_1 + \chi_{U^0} \Div{\bf v}_2 +
[\Trp{{\bf v}_1}{\partial^* U} - \Trm{{\bf v}_2}{\partial^* U}]\,
\Haus{N-1}\res\partial^*U.  
\]
\end{theorem}
\begin{proof} 
It will suffice to apply \cite[Theorem 5.3]
{ComiPayne} with $\hat{F}_i={\bf v}_i$, $i=1,2$. 
\end{proof}

By combining this gluing construction and the integration by parts formula \eqref{f:anz1}
we are now in a position to prove the following result. 

\begin{theorem}[Tensor Gauss--Green formula for weakly regular sets]\label{t:GGw}
Let
$\MA\in\DMAloc(\R^N;\mathbb{R}^{N\times N})$ and $\bfu\in BV_{\rm loc}(\R^N;\R^N)\cap L^\infty_{\rm loc}(\R^N;\R^N)$. Let $E\subset\R^N$ be a weakly regular set. 
Then the following Gauss--Green formula holds:
\begin{eqnarray}
& \displaystyle \int_{ E} \bfu^* \cdot d\DIV\MA + \int_{ E}(\MA:D\bfu)  =- \int_{\partial E} \bfu^+\cdot \Trpv{\MA}{\partial E} d\Haus{N-1}\,,\label{eq:gaussgreen145} 
\end{eqnarray}
where $ \Trpv{\MA}{\partial E}$
is the normal trace of \(\MA\) when \(\partial  E\) is oriented
with respect to the interior unit normal vector (which exists $\Haus{N-1}$-a.e. in $\partial E$, see \cite[Theorem 2.9]{ComiLeo}).
\end{theorem}

\begin{proof}
Since $ E$ is a set of finite perimeter,
it holds $\partial^* E \subseteq \partial E$,
hence \eqref{eq:weaklyreg} 
implies that
$\Haus{N-1}(\partial E\setminus\partial^* E) = 0$.
Consequently,
\begin{equation}\label{f:incl}
\Haus{N-1}\res \partial E = \Haus{N-1}\res \partial^* E,
\quad
\Haus{N-1}( E^1\setminus E) = 0,
\quad
\Haus{N-1}( E^0\setminus(\R^N\setminus\overline{ E})) = 0.
\end{equation}

Let us consider the vector field
\[
{\bf v}(x) :=\chi_ E(x)\MA^T(x)\bfu(x)=
\begin{cases}
{\bf v}_1(x) := \MA^T(x)\bfu(x), 
&\text{if}\ x\in E,\\
{\bf 0},
&\text{if}\ x\in \R^N\setminus E\,.
\end{cases}
\]
By \eqref{f:incl} we may apply Theorem \ref{t:gluing}, so that ${\bf v}\in \DMloc(\R^N)$ and it is compactly supported. Hence,
\[
0=\Div{\bf v}= \chi_{ E} \Div{\bf v}_1 + \Trp{{\bf v}_1}{\partial E}\,\Haus{N-1}\res\partial E.
\]
Therefore, by Proposition \ref{p:tracesb}, we obtain
\begin{equation}\label{nono}
\int_{ E} \, d\Div(\MA^T\bfu) =-\int_{\partial  E}\bfu^+\cdot\Trpv{\MA}{\partial E}\ d\Haus{N-1}.
\end{equation}
Now, we write the left-hand side of \eqref{nono} using the pairing $(\MA:D\bfu)$, and, applying the integration by parts formula \eqref{f:anz1}, we obtain
$$
\int_{ E} \, d(\bfu^* \cdot \DIV\MA + (\MA: D\bfu)) =-\int_{\partial  E}\bfu^+\cdot\Trpv{\MA}{\partial E}\ d\Haus{N-1},
$$
which corresponds to \eqref{eq:gaussgreen145}. 
\end{proof}

\section{Gauss--Green Formula up to the boundary on rough open sets}

In \cite{ChLiTo}, the authors consider the general class of \emph{rough} domains 
by imposing a weaker notion of regularity on $\Omega$. Specifically, the open set $\Omega$ is assumed to satisfy
\begin{equation}\label{tata}
\Haus{N-1}(\partial\Omega \setminus \Omega^0) < +\infty,
\end{equation}
where $\Omega^0$ denotes the measure-theoretic exterior of $\Omega$. It is shown (see \cite[Theorem 3.1]{ChLiTo}) that the sets satisfying \eqref{tata} are exactly those that can be approximated from the interior by smooth sets with uniformly bounded perimeter.  
Moreover, as proved in \cite[Theorem 4.2]{ChLiTo}, every bounded open set satisfying \eqref{tata} is an \emph{extension domain} for bounded divergence-measure fields. That is, for every $\A \in \DMloc(\R^N;\R^N)$, the extension 
\[
\widetilde \A(x) := \chi_\Omega(x) \A(x)
\] 
defines a divergence-measure field in $\R^N$ with $|\Div \widetilde \A|(\R^N) < +\infty$. Furthermore, open sets satisfying \eqref{tata} are also extension domains for bounded $BV$ functions (see \cite[Corollary 7.2]{ChLiTo}). Finally, each weakly regular set is a rough set, since
\begin{equation*}
\Haus{N-1}(\partial \Omega \setminus \Omega^0) \leq \Haus{N-1}(\partial \Omega) = \Haus{N-1}(\partial^* \Omega)<+\infty\,.
\end{equation*}

The following result is based on \cite[Theorem 5.2]{ChLiTo}.

\begin{theorem}
\label{t:GGrough}
Let
$\MA\in\DMAloc(\R^N;\mathbb{R}^{N\times N})$ and $\bfu\in BV_{\rm loc}(\R^N;\R^N)\cap L^\infty_{\rm loc}(\R^N;\R^N)$. Let $ E\subset\R^N$ be a bounded open set of finite perimeter satisfying \eqref{tata}. Then there exists $g\in L^\infty(\partial E\setminus  E^0;\Haus{N-1})$ such that the following Gauss--Green formula holds
\begin{align}\label{eq:gaussgreen1333} 
& \displaystyle \int_{ E} \bfu^* \cdot d\DIV\MA + \int_{ E}(\A:D\bfu)\\
=&\displaystyle - \int_{\partial^* E} \bfu^+\cdot \Trpv{\MA}{\partial^* E}\ d\Haus{N-1}
- \int_{\partial E\cap E^1} g(x)\ d\Haus{N-1},\nonumber
\end{align}
where $ \Trpv{\MA}{\partial^* E}$
is the normal trace of \(\MA\) when \(\partial^*  E\) is oriented
with respect to the interior unit normal vector.
\end{theorem}
\begin{proof}
Let ${\bf v}(x):=
({\MA}^T{\bfu})(x)$, and $\widetilde{\bf v}$ denote the extension of ${\bf v}$  to $\R^N$. Then $\widetilde{\bf v}\in \DMloc(\R^N;\R^N)$. By \cite[Theorem 5.2]{ChLiTo} there exists a trace $g\in L^\infty(\partial E\setminus  E^0;\Haus{N-1})$ such that 
$$
\int_{\partial E\setminus E^0} g(x) d\Haus{N-1}=
- 2\int_{\partial^* E} (\widetilde{\bf v}, D\chi_ E)
- \int_{\partial E\cap E^1}\,d\Div\widetilde{\bf v},
$$
where $(\widetilde{\bf v}, D\chi_ E)$ is the pairing defined in \eqref{f:pairing}.

Note that $\partial E\setminus E^0=(\partial E\cap E^1)\cup \partial^* E$ up to a $\Haus{N-1}$-null set.
By Proposition \ref{p:tracesb} we have
\begin{equation*}
\begin{split}
(\widetilde{\bf v}, D\chi_ E) & = \Tr^*(\widetilde{\bf v},\partial^* E)\Haus{N-1}\res\partial^* E \\
& =\frac12\Tr^+({\bf v},\partial^* E)\Haus{N-1}\res\partial^* E \\
& =\frac12\bfu^+\cdot \Trpv{\MA}{\partial^* E} \Haus{N-1}\res\partial^* E.
\end{split}
\end{equation*}
Moreover, choosing $\varphi\in C^1_c(\R^N)$ with $\varphi=1$ on $\overline E$ in \cite[Theorem 5.2]{ChLiTo}, we obtain the following
Gauss--Green formula up to the boundary:
\begin{equation*}\label{eq:gaussgreen13334} 
 \displaystyle \int_{ E} \,d\Div{\bf v}
=\displaystyle - \int_{\partial^* E} \bfu^+\cdot \Trpv{\MA}{\partial^* E} d\Haus{N-1}
- \int_{\partial E\cap E^1} g(x) d\Haus{N-1},
\end{equation*}
 where $g$ is the density of the measure $\Div\widetilde{\bf v}=\Div(\widetilde{\A^T\bfu})$ with respect to $\Haus{N-1}\res(\partial E\cap E^1)$.
Now, by using formula \eqref{f:anz1} and the new pairing $(\MA:D\bfu)$ for vector-valued $BV$ functions, %$\bfu$,  
the identity \eqref{eq:gaussgreen1333} follows.
\end{proof}

\begin{remark}
If $ E$ is weakly regular, then $\Haus{N-1}(\partial E\setminus\partial^* E) = 0$. In particular, this entails that $\Haus{N-1}(\partial E\cap E^1)=\Haus{N-1}(\partial E\cap E^0)=0$, so that \eqref{eq:gaussgreen1333} reduces to \eqref{eq:gaussgreen145}. 
\end{remark}

\def\cprime{$'$}

\end{document}